\documentclass[12pt]{article}
\usepackage{srcltx}
\usepackage{eurosym}
\usepackage{mathtools}
\usepackage{amsmath}
\usepackage{amsfonts}
\usepackage{amssymb}
\usepackage{amsthm}
\usepackage{graphicx}
\usepackage{mathrsfs}
\usepackage{xcolor}
\usepackage{exscale}
\usepackage{latexsym}
\usepackage{authblk}
\usepackage[colorlinks,plainpages=true,pdfpagelabels,hypertexnames=true,colorlinks=true,pdfstartview=FitV,linkcolor=blue,citecolor=red,urlcolor=black]{hyperref}
\PassOptionsToPackage{unicode}{hyperref}
\PassOptionsToPackage{naturalnames}{hyperref}
\usepackage{enumerate}
\usepackage[shortlabels]{enumitem}
\usepackage{bookmark}
\usepackage{wasysym}
\usepackage{esint}
\usepackage[ddmmyyyy]{datetime}
\usepackage[margin=1in]{geometry}
\makeatletter
\g@addto@macro\normalsize{%
	\setlength\abovedisplayskip{4pt}
	\setlength\belowdisplayskip{4pt}
	\setlength\abovedisplayshortskip{4pt}
	\setlength\belowdisplayshortskip{4pt}
}
\numberwithin{equation}{section}
\everymath{\displaystyle}
\usepackage[capitalize,nameinlink]{cleveref}
\crefname{section}{Section}{Sections}
\crefname{subsection}{Subsection}{Subsections}
\crefname{condition}{Condition}{Conditions}
\crefname{hypothesis}{Hypothesis}{Conditions}
\crefname{assumption}{Assumption}{Assumptions}
\crefname{lemma}{Lemma}{Lemmas}
\crefname{definition}{Definition}{Definitions}

\crefformat{equation}{\textup{#2(#1)#3}}
\crefrangeformat{equation}{\textup{#3(#1)#4--#5(#2)#6}}
\crefmultiformat{equation}{\textup{#2(#1)#3}}{ and \textup{#2(#1)#3}}
{, \textup{#2(#1)#3}}{, and \textup{#2(#1)#3}}
\crefrangemultiformat{equation}{\textup{#3(#1)#4--#5(#2)#6}}%
{ and \textup{#3(#1)#4--#5(#2)#6}}{, \textup{#3(#1)#4--#5(#2)#6}}%
{, and \textup{#3(#1)#4--#5(#2)#6}}

\Crefformat{equation}{#2Equation~\textup{(#1)}#3}
\Crefrangeformat{equation}{Equations~\textup{#3(#1)#4--#5(#2)#6}}
\Crefmultiformat{equation}{Equations~\textup{#2(#1)#3}}{ and \textup{#2(#1)#3}}
{, \textup{#2(#1)#3}}{, and \textup{#2(#1)#3}}
\Crefrangemultiformat{equation}{Equations~\textup{#3(#1)#4--#5(#2)#6}}%
{ and \textup{#3(#1)#4--#5(#2)#6}}{, \textup{#3(#1)#4--#5(#2)#6}}%
{, and \textup{#3(#1)#4--#5(#2)#6}}

\crefdefaultlabelformat{#2\textup{#1}#3}

\newtheorem{theorem} {Theorem}[section]
\newtheorem{proposition} [theorem]{Proposition}

\newtheorem{lemma}[theorem]{Lemma}

\newtheorem{counter example}[theorem]{Counter Example}
\newtheorem{remark}[theorem] {Remark}

\def\N{\mathbb{N}}
\def\CC{{\rm \kern.24em \vrule width.02em height1.4ex depth-.05ex \kern-.26emC}}

\def\TagOnRight

\def\AA{{it I} \hskip-3pt{\tt A}}

\def\QQ{\rlap {\raise 0.4ex \hbox{$\scriptscriptstyle |$}} {\hskip -0.1em Q}}

\makeatletter
\newcommand{\vo}{\vec{o}\@ifnextchar{^}{\,}{}}
\makeatother

\def\YYint#1#2#3{{\setbox0=\hbox{$#1{#2#3}{\iint}$}
		\vcenter{\hbox{$#2#3$}}\kern-.50\wd0}}

\def\XXint#1#2#3{{\setbox0=\hbox{$#1{#2#3}{\int}$}
		\vcenter{\hbox{$#2#3$}}\kern-.50\wd0}}

\makeatletter
\def\namedlabel#1#2{\begingroup
	\def\@currentlabel{#2}%
	\label{#1}\endgroup
}
\makeatother
\makeatletter
\newcommand{\rmh}[1]{\mathpalette{\raisem@th{#1}}}
\newcommand{\raisem@th}[3]{\hspace*{-1pt}\raisebox{#1}{$#2#3$}}
\makeatother

\newcounter{desccount}
\newcommand{\descitem}[2]{\item[#1]\refstepcounter{desccount}\label{#2}}
\newcommand{\descref}[2]{\hyperref[#1]{\textnormal{\textcolor{black}{}\textcolor{blue}{ #2}\textcolor{black}{}}}}

\newcommand{\dref}[2]{\hyperref[#1]{\textcolor{black}{(}\textcolor{blue}{\bf #2}\textcolor{black}{)}}}
\newcommand{\be} {\begin{eqnarray}}
	\newcommand{\ee} {\end{eqnarray}}
\newcommand{\Bea} {\begin{eqnarray*}}
	\newcommand{\Eea} {\end{eqnarray*}}
\newcommand{\pa} {\partial}

\newcommand{\al} {\alpha}
\newcommand{\rr}{\rightarrow}

\newcommand{\B} {\beta}
\newcommand{\de} {\delta}

\newcommand{\p}  {\prime}
\newcommand{\e}  {\varepsilon}
\newcommand{\De} {\Delta}

\newcommand{\la} {\lambda}
\newcommand{\si} {\sigma}

\newcommand{\La} {\Lambda}

\newcommand{\f}{\infty}
\newcommand{\R}{\mathbb{R}}

\newcommand{\norm}[1]{\left|\hspace{-0.2mm}\left| #1 \right|\hspace{-0.2mm}\right|}
\newcommand{\abs}[1]{\left| #1\right|}

\newcounter{whitney}
\refstepcounter{whitney}

\newcounter{ineqcounter}
\refstepcounter{ineqcounter}
\makeatletter
\def\ps@pprintTitle{%
	\let\@oddhead\@empty
	\let\@evenhead\@empty
	\def\@oddfoot{}%
	\let\@evenfoot\@oddfoot}
\makeatother
\usepackage[doublespacing]{setspace}
\usepackage[titletoc,toc,page]{appendix}

\makeatletter
\newcommand{\refcheckize}[1]{%
	\expandafter\let\csname @@\string#1\endcsname#1%
	\expandafter\DeclareRobustCommand\csname relax\string#1\endcsname[1]{%
		\csname @@\string#1\endcsname{##1}\wrtusdrf{##1}}%
	\expandafter\let\expandafter#1\csname relax\string#1\endcsname
}
\makeatother

\refcheckize{\cref}
\refcheckize{\Cref}


\makeatletter
\newcommand{\mainsectionstyle}{%
	\renewcommand{\@secnumfont}{\bfseries}
	\renewcommand\section{\@startsection{section}{2}%
		\z@{.5\linespacing\@plus.7\linespacing}{-.5em}%
		{\normalfont\bfseries}}%
}
\makeatother
\usepackage{pgf,tikz}
\usetikzlibrary{arrows}
\usetikzlibrary{decorations.pathreplacing}

\usepackage{xpatch}
\xpatchcmd{\MaketitleBox}{\hrule}{}{}{}
\xpatchcmd{\MaketitleBox}{\hrule}{}{}{}

\date{}

\usepackage{scalerel}

\makeatletter

\title{Existence of BV solutions for $2\times2$ hyperbolic balance laws for $L^\f$ initial data}

\author[1,a]{Boris Haspot}
\author[2,a]{Animesh Jana}

\affil[a]{\footnotesize	 Universit\'e Paris Dauphine, PSL Research University, CEREMADE (UMR CNRS 7534), Place du Mar\' echal De Lattre De Tassigny 75775 Paris cedex 16 (France).}

\affil[1]{\em \footnotesize	 haspot@ceremade.dauphine.fr}
\affil[2]{\em \footnotesize	 animesh.jana@dauphine.psl.eu}

\allowdisplaybreaks

\begin{document}
	\maketitle
\begin{abstract}
	We prove the existence of BV solutions for $2\times 2$ system of hyperbolic balance laws in one space dimension. The flux is assumed to have two genuinely nonlinear characteristic fields. We consider a general force which may possibly depend on time and space variable as well. To prove the existence, we assume the initial data to be small in $L^\f$. Furthermore, we also study qualitative behavior for entropy solutions to hyperbolic system of balance laws.  
\end{abstract}	
	\tableofcontents
	
	\section{Introduction}
	We consider the Cauchy problem  corresponding to a $2\times 2$ hyperbolic balance laws
	\begin{align}
		u_t+f(u)_x&=g(t,x,u),\,\,\,\mbox{ for }(t,x)\in\R_+\times\R,\label{eqn-main}\\
		u(0,x)&=\bar{u}(x),\quad\mbox{ for }x\in\R.\label{eqn-data}
	\end{align}
	In this article, we impose the following condition on the flux function $f$. 
	\begin{description}
		\descitem{($\mathcal{H}_f$)}{H-f} $f:B(0,r)\rr\R^2$ for some $r>0$ is smooth and $Df(0)$ is strictly hyperbolic with both genuinely nonlinear characteristic fields. Moreover, we consider
		\begin{equation}\label{def:strict-hyperbolicity}
			\la_1(u)\leq -\frac{c_0}{2}<0<\frac{c_0}{2}\leq \la_2(u)\mbox{ for all }u\in B(0,r),
		\end{equation}
		for some $c_0>0$.
		
	\end{description}
	
	In addition, we assume the force $g$ verifies the following two properties.
	\begin{description}
		
		\descitem{($\mathcal{H}_g$1.)}{H-g-1} Let $T^*>0$ and $g:[0,T^*)\times\R\times B(0,r)\rr\R^2$ is measurable w.r.t. $t$ and for any $x\in\R$, $u\in B(0,r)$ and $(x,u)\mapsto g(t,x,u)$ is a $C^2$ function for each $t\in [0,T^*)$. Furthermore, we assume that $\sup\limits_{t\in [0,T^*)}\norm{g(t,\cdot,\cdot)}_{C^2}<+\f$.

		\descitem{($\mathcal{H}_g$2.)}{H-g-2} There exist two functions $\omega_1\in L^1(\R)\cap L^\f(\R)$ and $\omega_2\in L^1([0,T^*))\cap L^\f([0,T^*))$ such that 
		\begin{equation}
			\abs{g(t,x,u)}+\norm{D_ug(t,x,u)}\leq \omega_2(t)\mbox{ and }\abs{g_x(t,x,u)}\leq \omega_1(x)\omega_2(t),
		\end{equation}
	 for all $t\in [0,T^*)$, $x\in\R$ and $u\in B(0,r)$.
	\end{description}

	In this article we prove the global existence of weak entropy solution for the Cauchy problem \eqref{eqn-main}--\eqref{eqn-data} when the initial condition in small in $L^\f$. More precisely, we show the following result. 
	\begin{theorem}\label{theorem-main}
		Let $f$ and $g$ be satisfying the assumption \descref{H-f}{($\mathcal{H}_f$)} and \descref{H-g-1}{($\mathcal{H}_g$1.)}--\descref{H-g-2}{($\mathcal{H}_g$2.)} respectively for some $T^*>0$. Then there exists small constant $\eta>0$ such that for any initial data $\bar{u}\in L^1_{loc}(\R;\R^2)$ and force $g$ satisfying 
		\begin{equation}\label{smallness-condition-1}
			\norm{\bar{u}}_{L^\f}+\norm{\omega_1}_{L^1(\R)}+\norm{\omega_2}_{L^1(0,T^*)}\leq\eta,
		\end{equation}
		the Cauchy problem admits a weak entropy solution for $t\in(0,T^*)$.
	\end{theorem}
    \begin{remark}
    	\begin{enumerate}
    		\item We note that $T^*$ can be taken as $T^*=+\f$ to obtain the global weak solutions in the case when $g$ satisfies time integrability in $(0,\f)$.
    		\item For $g(t,x,u)=g(u)$, we take $\omega_1\equiv0$ and $\omega_2\equiv C$ for some $C>0$ and in this case we can apply Theorem \ref{theorem-main} with $T^*<\frac{\eta}{2C}$ and $\norm{\bar{u}}_{L^\f}\leq \eta/2$.
    	\end{enumerate}
    \end{remark}

For $n\times n$ hyperbolic system, existence of BV solutions is first established by Glimm \cite{Glimm} for initial data with small total variation. See also \cite{Bressan-book} for existence via the front tracking approximation. For large total variation, the existence of BV solutions are limited to special systems. For $2\times 2$ hyperbolic system of conservation laws, local existence of BV solutions from initial data with large total variation has been shown by Bressan and Colombo \cite{BrCo-uniq}.

Glimm and Lax \cite{Glimm-Lax} have shown the existence of entropy solution for $2\times 2$ system when both fields are genuinely nonlinear and each field satisfies a condition (see section \ref{section:L-infty} for a brief discussion on this direction). Later, Bianchini et al \cite{Bi-Col-Mon} improved the result for any $2\times 2$ system with genuinely nonlinear fields. 

One of the key tool for proving the BV estimate in positive time from $L^\f$ initial data is establishing Oleinik type estimate. We note that even for scalar balance laws the BV regularizing is limited to uniformly convex flux (see \cite{GGJJ,Lax-57,Oleinik} and references therein).  In the context of hyperbolic system, existence of solutions for a non-BV initial data is less studied. We refer to \cite{BreGoa-jde} where the existence of entropy solutions for non-BV initial data is established for $n\times n$ Temple system by establishing Oleinik type estimates when all fields are genuinely nonlinear. In \cite{Haspot-Junca} the existence of entropy solutions has been proved when data has fractional BV regularity when flux has one genuinely nonlinear field and one linearly degenerate field. For a more general class of $2\times 2$ system, global existence of entropy solutions has been shown in \cite{Glass-24} for initial data having fractional BV regularity. Even for these systems the BV regularizing is limited to the assumption of both fields are genuinely nonlinear (see \cite{BCGJ,GGJJ}).

For balance laws, the existence of BV solutions are mostly studied only for initial data with small total variation. Amadori {\em et al} \cite{AGG} proved global existence of BV solutions when the force $g$ is integrable in $x$ variable. On the hand, for more general force term, local existence of BV solution is know due to Colombo and Guerra \cite{CG07}. We refer to \cite{AGG,AG02} for uniqueness and stability results for entropy solutions for balance laws. We note that in the case of $L^\f$ initial data, the stability remains an open question even for the general $2\times2$ system of conservation laws with both genuinely nonlinear fields.  

For system of conservation laws, qualitative properties of entropy solutions has been studied in \cite{BreLef,Bressan-book}. The decay of positive waves has been shown in \cite{BreLef}. For BV entropy solution, it has been shown that the discontinuities lie in countable family of Lipschitz graphs. For balance laws, the study is limited to special choices of force (see for instance \cite{ACM-19}). The decay of positive waves for balance laws can be found in \cite{GG-04}. We have shown that for a general form of force term $g$, the discontinuities of entropy solutions lie on a countable union of Lipschitz graphs.

Rest of the paper is organized as follows. In the next section, we recall few basic results from and describe the approximation method combining wave front tracking with operator splitting. We prove BV estimates in section when initial data has bounded total variation. The proof of Theorem \ref{theorem-main} is presented in section \ref{section:L-infty} by establishing the $L^\f$ bounds. In the last section \ref{sec:structure}, we study the qualitative properties of solutions.

\section{Outline of the proof}
	Following \cite{Bi-Col-Mon}, we can assume the flux in a simplified form and $\la_1(0)=-1,\la_2(0)=1$. By a change of variable we can assume that $f$ as the following form,
	\begin{equation}\label{formation-f}
		\begin{array}{rl}
			f_1(u)&=-u_1+\frac{\al_{11}}{2}u_1^2+\al_{12}u_1u_2+\frac{\al_{22}}{2}u_2^2+\mathcal{O}(1)\norm{u}^3,\\
			f_2(u)&=u_2+\frac{\B_{11}}{2}u_1^2+\B_{12}u_1u_2+\frac{\B_{22}}{2}u_2^2+\mathcal{O}(1)\norm{u}^3,
		\end{array}
	\end{equation}
	with $\al_{ij}:=\frac{\pa^2f_1}{\pa u_i\pa u_j}(0)$ and $\B_{ij}:=\frac{\pa^2f_2}{\pa u_i\pa u_j}(0)$.  We define $\hat{\la}$ as the maximum speed of characteristics defined as follows
	\begin{equation}\label{def:max-speed}
		\hat{\la}:=\sup\limits_{u\in B(0,r)}\{\abs{\la_i(u)};\,i=1,2\}.
	\end{equation}
	We note that a $2\times 2$ hyperbolic system admits a pair of Riemann invariants $(v_1,v_2)$ and there exists a invertible smooth map $\mathcal{J}:(u_1,u_2)\mapsto (v_1,v_2)$. Next we consider Riemann problem for the following homogeneous equation
	\begin{equation}\label{eqn-homogeneous}
		u_t+f(u)_x=0.
	\end{equation} 
	We consider Riemann data at $x_0$
	\begin{equation}\label{data-Riemann}
		\bar{u}_{Rie}(x)=\left\{\begin{array}{rl}
			u_l&\mbox{ for }x<x_0,\\
			u_r&\mbox{ for }x>x_0
		\end{array}\right.
	\end{equation} 
	%
	%
	%
	%
	%
	Let us recall from \cite{BrCo-semi} an approximation of solution. We denote the $i$-th rarefaction curve as $\varphi_i^+$ and $i$-th shock curve as $\varphi^-_i$ for $i=1,2$. We have
	\begin{align}
		\varphi^+_1(\si;v_1,v_2)=(v_1+\si,v_2),&\quad 	\varphi^-_1(\si;v_1,v_2)=(v_1+\si,v_2+\hat{\varphi}_2(\si,v_1,v_2)\si^3),\\
		\varphi^+_2(\si;v_1,v_2)=(v_1,v_2+\si),&\quad 	\varphi^-_2(\si;v_1,v_2)=(v_1+\hat{\varphi}_1(\si,v_1,v_2)\si^3,v_2+\si).
	\end{align}
	Consider a $C^\f$ function $\xi$ such that
	\begin{equation}
		\left\{\begin{array}{cl}
			\xi(y)=1&\mbox{ if }y\leq -2,\\
			-2\leq \xi^\p(y)\leq 0&\mbox{ if }y\in[-2,-1],\\
			\xi(y)=0&\mbox{ if }y\geq-1.
		\end{array}\right.
	\end{equation}
	Following \cite{BrCo-semi} we define for $\e>0$,
	\begin{equation}
		\psi^\e_i(\si;v_1,v_2)=\xi\left(\frac{\si}{\sqrt{\e}}\right)\varphi^-_i(\si;v_1,v_2)+\left(1-\xi\left(\frac{\si}{\sqrt{\e}}\right)\right)\varphi^+_i(\si;v_1,v_2)\mbox{ for }i=1,2.
	\end{equation}
	Let $u_m$ be the middle state. Denote $v_m=v^m_1,v^m_2)=\mathcal{J}(u_m)$. Then we have
	\begin{equation*}
		(v^r_1,v^r_2)=\psi^\e_2(\si_2;v_1^m,v_2^m)\mbox{ and }(v_1^m,v_2^m)=\psi^\e_1(\si_1;v_1^l,v_2^l).
	\end{equation*}
	If $\si_1\geq0$, there exist $j,k$ such that
	\begin{equation*}
		j\e\leq v_1^l<(j+1)\e\mbox{ and }k\e\leq v_1^m<(k+1)\e.
	\end{equation*}
	We introduce the states $\bar{v}_{i,1}:=i\e$ for $j\leq i\leq k+1$. Following \cite{BrCo-semi} we approximate the rarefaction fan as follows
	\begin{equation}\label{def:rarefaction-partition}
		(v_1^\e(t,x),v_2^\e(t,x))=\left\{\begin{array}{ll}
			(v_1^l,v_2^l)&\mbox{ if }x<\la_1(\bar{v}_{j,1},v_2^l)t,\\
			(\bar{v}_{i,1},v^l_2)&\mbox{ if }\la_1(\bar{v}_{i,1},v_2^l)t<x<\la_1(\bar{v}_{i+1,1},v_2^l)t,\,j+1\leq i\leq k,\\
			(v_1^m,v_2^l)&\mbox{ if }\la_1(\bar{v}_{k,1},v_2^l)t<x<0.
		\end{array}\right.
	\end{equation}
	When $\si_1<0$, it is connected by 1-shock, we define
	\begin{equation}
		(v_1^\e(t,x),v_2^\e(t,x))=\left\{\begin{array}{ll}
			(v_1^l,v_2^l)&\mbox{ if }x<\la_1^\xi(\si_1;v_1^l,v_2^l)t,\\
			(v_1^m,v_2^m)&\mbox{ if }\la^\xi_1(\si_1;v_1^l,v_2^l)t<x<0,
		\end{array}\right.
	\end{equation}
	where $\la_1^\xi(\si_1;v_1^l,v_2^l)$ is defined as
	\begin{equation}
		\la_1^\xi(\si_1;v_1^l,v_2^l):=\xi\left(\frac{\si_1}{\sqrt{\e}}\right)\la_1^s(\si_1;v_1^l,v_2^l)+\left(1-\xi\left(\frac{\si_1}{\sqrt{\e}}\right)\right)\la_1^r(\si_1;v_1^l,v_2^l),
	\end{equation}
	where $\la^{s}_1$ and $\la_1^{r}$ are defined as 
	\begin{align*}
		\la_1^s(\si_1;v_1^l,v_2^l)&:=\la_1((v_1^l,v_2^l);\varphi^-_1(\si_1;v_1^l,v_2^l)),\\
		\la_1^r(\si_1;v_1^l,v_2^l)&:=	\sum\limits_{i}\frac{\mathcal{L}^1([i\e,(i+1)\e]\cap[v_1^m,v_1^l])}{\abs{\si_1}}\la_1(\hat{v}_{i,1},v_2^l).
	\end{align*}
	We consider a similar approximation for 2-waves as well. Now, for a Riemann data \eqref{data-Riemann} with $x_0=0$, we consider the unique approximated solution $u^\e$ consisting the waves approximated as above. See \cite{BrCo-semi} for more details.

	\subsection{Algorithm}
	We consider two sequences
	\begin{equation}
		\{\tau_\nu\}_{\nu\in\N},\,\, \{\e_\nu\}_{\nu\in\N},\,\,0<\tau_\nu\leq \e_\nu\rr0.
	\end{equation}
We set
\begin{equation}
	x_j:=j\e_\nu\mbox{ for }j\in\mathbb{Z}\mbox{ and }t_n=n\tau_\nu\mbox{ for }n\in\N.
\end{equation}
Consider a piecewise constant function $\bar{u}_\nu$ satisfying 
	\begin{equation}
		TV(\bar{u}_\nu)\leq TV(\bar{u})\mbox{ and }\norm{\bar{u}_\nu-\bar{u}}_{L^1}\rr0\mbox{ as }\nu\rr\f.
	\end{equation}
We discretize $g(t,x,u)$ as follows, 
\begin{equation}
	g_\nu(t,x,u):=\sum\limits_{n\geq 1}\sum\limits_{j\in\mathbb{Z}}\chi_{[x_j,x_{j+1})}(x)\chi_{[t_{n-1},t_n)}(t)g_{n,j}(u),
\end{equation}
where $ g_{n,j}(u)$ is defined as
\begin{equation}
      g_{n,j}(u)=\frac{1}{\tau_\nu}\int\limits_{t_{n-1}}^{t_n}\frac{1}{\e_\nu}\int\limits_{x_j}^{x_{j+1}}g(s,y,u)\,dyds.
\end{equation}
Furthermore, we consider discretization of $\omega_1,\omega_2$ as follows,
\begin{align}
	\omega_1^\nu(x)&=\sum\limits_{j\in\mathbb{Z}}\chi_{[x_j,x_{j+1})}(x)\omega_{1,j},\mbox{ where }\omega_{1,j}=\frac{1}{\e_\nu}\int\limits_{x_j}^{x_{j+1}}\omega_1(y)\,dy,\\
	\omega_2^\nu(t)&=\sum\limits_{n\geq1}\chi_{[t_{n-1},t_n)}(t)\omega_{2,n},\mbox{ where }\omega_{2,j}=\frac{1}{\tau_\nu}\int\limits_{t_{n-1}}^{t_n}\omega_2(s)\,ds.
\end{align}
We note that
\begin{equation}
	g_\nu(t,x,u)\leq \omega_1^\nu(x)\omega_2^\nu(t)\mbox{ for all }t>0\mbox{ and }x\in\R.
\end{equation}
\begin{enumerate}
	\item We solve by wave front tracking up to time $\tau_\nu$ and obtain $\e_\nu$-approximate front tracking solution.
	\item At $t=\tau_\nu$, we define
	\begin{equation}
		u_\nu(\tau_\nu +,\cdot)=u_\nu(\tau_\nu -,\cdot)+\tau_\nu g_\nu(\tau_\nu-,\cdot, u_\nu(\tau_\nu-,\cdot)).
	\end{equation}
   \item Having defined $u_\nu(n\tau_\nu+,\cdot)$, we again use the front tracking algorithm until time $t=(n+1)\tau_\nu$.
   \item At $t=(n+1)\tau_\nu$, we define
   \begin{equation}\label{defn:time-step}
   	u_\nu((n+1)\tau_\nu+,\cdot)=u_\nu((n+1)\tau_\nu-,\cdot)+\tau_\nu g_\nu((n+1)\tau_\nu-,\cdot, u_\nu((n+1)\tau_\nu-,\cdot)).
   \end{equation}
\end{enumerate}
By changing speed of the fronts slightly we consider that no two interactions occur at same time for $t\neq t_n$, at any interaction point only two waves meet and at $t=t_n$ there is no interaction between two waves. 

Furthermore, we have taken the following provisions in the construction. Consider a situation when an $1$-wave $\si_1>0$ interact with another wave $\si$. Let $u_l,u_m,u_r$ be the states appearing in the outgoing waves. If $\si_1^+$ is the outgoing $1$-wave. When $\si_1^+>0$ we write a single wave with speed $\la_!(u_m)$ instead of a possible rarefaction fan. See \cite{Bressan-book}. At $t=t_n$ if there is a 1-wave $\si_1^->0$, we do the same procedure, that is, if the outgoing states are $u_l,u_m,u_r$ and the strength of 1-wave is $\si_1^+>0$, we write only single 1-wave with left state $u_l$ and right state $u_m$ and speed $\la_1(u_m)$. We follow the same procedure for 2- rarefaction waves as well.

 \begin{remark}
 	We note that at interaction point $(t,x)$ with $t\neq t_n$ the newly created rarefaction waves are solved as in \eqref{def:rarefaction-partition}. Similarly, at $t=t_n$ the newly created rarefaction waves are written as in \eqref{def:rarefaction-partition}. 
 \end{remark}
Since we have not divided the resultant rarefaction waves in some cases which mean that the positive waves are not necessarily bounded by $\epsilon_\nu$ but we will show that the strength remains of the order $\mathcal{O}(\e_\nu)$ (see Proposition \ref{prop:rarefaction}).

	\section{BV estimates}\label{sec:construction}
	\begin{lemma}[\cite{BrCo-semi}]\label{lemma-1}
		There exists $C_1>0$ such that the following holds true.
		\begin{enumerate}[(i.)]
			\item If a 2-wave $\si_2^-$ interacts with 1-wave $\si_1^-$ and $\si_1^+,\si_2^+$ are outgoing waves then,
			\begin{equation*}
				\abs{\si_1^+-\si_1^-}+\abs{\si^+_2-\si_2^-}\leq C_1\abs{\si_1^-\si_2^-}(\abs{\si_1^-}+\abs{\si_2^-}).
			\end{equation*} 
			\item If two 1-waves $\si^\p,\si^{\p\p}$ interact and produce outgoing waves $\si_1^+,\si_2^+$ then,
			\begin{equation*}
				\abs{\si_1^+-(\si^\p+\si^{\p\p})}+\abs{\si^+_2}\leq C_1\abs{\si^\p\si^{\p\p}}(\abs{\si^\p}+\abs{\si^{\p\p}}).	
			\end{equation*}
			\item If two 2-waves $\si^\p,\si^{\p\p}$ interact and produce outgoing waves $\si_1^+,\si_2^+$ then,
			\begin{equation*}
				\abs{\si^+_1}+\abs{\si_2^+-(\si^\p+\si^{\p\p})}\leq C_1\abs{\si^\p\si^{\p\p}}(\abs{\si^\p}+\abs{\si^{\p\p}}).	
			\end{equation*}
		\end{enumerate}
	\end{lemma}
\begin{lemma}
	Let $\hat{u}_l=u_l+\tau g_\nu(t_n-,x_l,w_l)$ and $\hat{u}_r=u_r+\tau g_\nu(t_n-,x_r,w_r)$. Suppose that there exist $\si,\hat{\si}$ such that $u_r=\Phi(\si)[u_l]$ and and $\hat{u}_r=\Psi(\hat{\si})[\hat{u}_l]$. Then we have
	\begin{equation}\label{interaction-estimate-force}
		\abs{\hat{\si}-\si}\leq C_2\tau \omega_2^\nu(t_n-) \left[\abs{\si}+\norm{\omega_{1}^\nu}_{L^1(x_l,x_r)}+\abs{w_l-w_r}\right].
	\end{equation}
\end{lemma}
\begin{proof}
	Let $E$ be defined as follows
	\begin{equation}
		E(\si,d,\tau):=\hat{\si}-\si\mbox{ where }d:=g_\nu(t_n-,x_r,w_r)-g_\nu(t_n-,x_l,w_l).
	\end{equation}
Observe that
\begin{equation}
	E(0,0,\tau)=E(\si,d,0)=0,\mbox{ and }\pa_dE(\si,0,0)=0,\pa_\tau E(0,0,\tau)=0.
\end{equation}
Then we have
\begin{align*}
	E(\si,d,\tau)&=E(\si,d,\tau)-E(\si,0,0)\\
	&=\int_0^{1}\left[d\pa_dE(\si,\theta d, \theta \tau)+\tau\pa_\tau E(\si,\theta d, \theta \tau)\right]\,d\theta\\
	&=d\int_0^{1}\left[\pa_dE(\si,\theta d, \theta \tau)-\pa_dE(\si,\theta d, 0)\right]\,d\theta\\
	&+\tau\int_0^{1}\left[\pa_\tau E(\si,\theta d, \theta \tau)-\pa_\tau E(0,0, \theta \tau)\right]\,d\theta\\
	&=d\tau \int_0^{1}\int_0^{1}\theta\pa_{d\tau}E(\si,\theta d, \xi\theta \tau)\,d\xi d\theta\\ 
	&+\tau\int_0^{1}\left[\si\pa_{\si \tau} E(\xi\si,\xi\theta d, \theta \tau)+\theta d\pa_{d\tau} E(\xi\si,\xi\theta d, \theta \tau)\right]\,d\xi d\theta\\
	&=\mathcal{O}(1)\tau \left[\abs{d}+\abs{\si}\right].
\end{align*}
We note that
\begin{align*}
	\abs{d}&\leq \abs{g_\nu(t_n-,x_r,w_r)-g_\nu(t_n-,x_r,w_l)}+\abs{g_\nu(t_n-,x_r,w_l)-g_\nu(t_n-,x_l,w_l)}\\
	&\leq \omega_{2,n}\left[\abs{w_l-w_r}+\norm{\omega_1^\nu}_{L^1(x_l,x_r)}\right].
\end{align*}
Hence the inequality \eqref{interaction-estimate-force} follows.
\end{proof}

	To estimate the $L^\f$ bound of solutions we use generalized characteristics corresponding to each field. Fix $X\in\R$ and $T>0$. For $i=1,2$ we consider a Lipschitz curve $x_i(t)$ for $i$-th field, 
	\begin{equation}
		\left.\begin{array}{rl}
			&\dot{x}_i(t)\in [	\underline{\la}_i(t,x),	\overline{\la}_i(t,x)],\\
			&x_i(T)=X,
		\end{array}\right\}
	\end{equation}
	where $\underline{\la}_i,\overline{\la}_i$ are defined 
	\begin{align*}
		\underline{\la}_i(t,x)&:=\min\{\la_i(u^\e(t,x-)),\la_i(u^\e(t,x+))\},\\
		\overline{\la}_i(t,x)&:=\max\{\la_i(u^\e(t,x-)),\la_i(u^\e(t,x+))\}.
	\end{align*}
	Now, we define {\em minimal backward characteristic} $y_i(t)$ emanating from $(T,X)$ as follows,
	\begin{equation}
		y_i(t):=\min\{x(t):\,\, x\mbox{ is a generalized $i$-th characteristic through $(T,X)$}\}.
	\end{equation}
We make the following observations for $i=1$ case.
\begin{enumerate}
	\item Due to the Lax entropy condition, $y_i(t)$ never crosses any 1-wave $\si_1<0$. Similarly, $y_1(t)$ can not pass through an interaction involving 1-wave with strength $\si_1<-\sqrt{\epsilon_\nu}$. The curve $y_1(t)$ can coincide with a 1-wave on a non-trivial interval of time when $\si_1\ge0$. As we consider the convention that for a BV function, $u(x)=u(x-)$, in the later case by $v_1(t,y_1(t))$ we mean $v_1(t,y_1(t))=v_1(t,y_1(t)-)$.
	\item We also note that any $y_1(t)$ can leave (as $t$ increases) from the 1-wave of strength $\si_1\ge-{\epsilon_\nu}$ in the left (as in the Figure \ref{fig-1}). Since $v_1(t,y_1(t))=v_1(t,y_1(t)-)$ when $y_1$ leave the front $v_1(t,y_1(t))$ does not change.

	\item We note that if $y_1(t)$ crosses an interaction point between a 2-wave $\si_2^-\geq -\epsilon_\nu$ and a 1-wave $\si_1^-\geq -\epsilon_\nu$, then $v_1(t,y_1(t))$ does not change.
	\item Suppose $y_1(t)$ crosses an interaction point of two 1-waves $\si^\p,\si^{\p\p}\geq -\epsilon_\nu$ and the outgoing 1-wave has positive strength. Then we note that $\De  v_1(t,y_1(t))=0$. To see this let us assume that $v^L,v^M,v^R$ be the states before interaction, then the out going wave is linear wave between $v^L$ and $v^R$. By the convention we have $v_1(t+,y_1(t+))=v^L_1=v_1(t-,y_1(t-))$.
\end{enumerate}	
	
	\begin{figure}[ht]
		\centering
		\begin{tikzpicture}

			\draw[thin][color=black] (-2.8 ,1) -- (-1,2.8) ;
			\draw[thin][color=black] (-2.8 ,1) -- (-0.6,2.8) ;
			\draw[thin][color=black] (-2.8 ,1) -- (-1.3,2.8) ;
			\draw[thin][color=black] (-2.8 ,1) -- (-4.5,2.5) ;
			
			\draw[thick,dashed][color=black] (-3.5,-0.5) -- (-2.8 ,1) -- (-2.05 ,1.9) -- (-1.8,2.9);
			
			\draw[thin][color=black] (-2.8 ,1) -- (-1,-0.2) ;
			\draw[thin][color=black] (-2.8 ,1) -- (-1.4,-0.7) ;
			
			\draw[thick] (-2.8,1) node[anchor=east] {\tiny$(t_0,x_0)$};
			
			\draw[thick] (-3.3,-0.2) node[anchor=east] {\tiny$y_2(t)$};
			
			\draw[thick] (-0.7,0.2) node[anchor=east] {\tiny$\si^{\p\p}$};
			\draw[thick] (-1.5,-0.6) node[anchor=east] {\tiny$\si^{\p}$};
			
			\draw[thick] (-4.5,2.5) node[anchor=east] {\tiny$\si^+_1$};
			\draw[thick](-1,2.8)  node[anchor=south] {\tiny$\si^+_2$};
			
			\draw[thick][color=black] (1.2 ,1) -- (5.9,1) ;
			
			\draw[thin][color=black] (3.5 ,1) -- (1.4,2.8) ;
			\draw[thin][color=black] (3.5 ,1) -- (1.9,2.8) ;
			\draw[thin][color=black] (3.5 ,1) -- (2.3,2.8) ;
			\draw[thin][color=black] (3.5 ,1) -- (5.5,2.8) ;
			
			\draw[thick,dashed][color=black] (4,-0.5) -- (3.5 ,1) -- (2.9 ,1.9) -- (2.7,2.9);
			
			\draw[thin][color=black] (3.5 ,1) -- (2,-0.8) ;
			
			\draw[thick] (1.2 ,1) node[anchor=south] {\tiny$t=t_n$};
			
			\draw[thick] (4,-0.4) node[anchor=west] {\tiny$y_1(t)$};
			
			\draw[thick] (2.2,-0.6) node[anchor=east] {\tiny$\si_-$};
			
			\draw[thick] (5.5,2.8) node[anchor=east] {\tiny$\si^+_2$};
			\draw[thick] (1.9,2.8)  node[anchor=south] {\tiny$\si^+_1$};

			\draw[thick] (-2.8,-1.75) node[anchor=south] {Fig 1a};
			\draw[thick] (3.5,-1.75) node[anchor=south] {Fig 1b};

		\end{tikzpicture}
		\caption{In Fig 1a, it illustrates at the point $(t_0,x_0)$ two 1-waves $\si^\p,\si^{\p\p}$ interacts and the outgoing waves are $\si_1^+$ and $\si_2^+$. In this situation, the dotted line denotes a possible minimal backward characteristic of 2-family. in Fig 1b, a 1-wave $\si_-$ changes at $t=t_n$ due to the effect of ODE. The outgoing waves are $\si_1^+,\si_2^+$. In this situation, the dotted line denotes a possible minimal backward characteristic of 1-family.}
		\label{fig-1}
	\end{figure}

	\begin{lemma}\label{lemma:change-vi}
		Let $t>0$ be such that $v_1(t+,y_1(t+))\neq v_1(t-,y_1(t-))$. Then, one of the following is true
		\begin{enumerate}
			\item $y_1$ crosses a 2-wave $\si_2$,
			\begin{equation}\label{change-v_1-1}
				\abs{v_1(t+,y_1(t+))}-\abs{v_1(t-,y_1(t-))}\leq C_3\abs{\si_2}^3.
			\end{equation}
			\item $y_1$ goes through an interaction point of two waves $\si^\p,\si^{\p\p}$ of second family and,
			\begin{equation}\label{change-v_1-2}
				\abs{v_1(t+,y_1(t+))}-\abs{v_1(t-,y_1(t-))}\leq C_3(\abs{\si^\p}+\abs{\si^{\p\p}})^3.
			\end{equation}

		\end{enumerate}
	\end{lemma}
\begin{proof}
	We note that \eqref{change-v_1-1} follows from the definition of $S_2$ curve. To prove the estimate \eqref{change-v_1-2}, let $v^l,v^m,v^r$ be the three states before interaction at $(t_0,y_1(t_0))$. Let $\hat{v}^m$ be the middle state arising in the Riemann problem solution for $(v^l,v^r)$. Let $\si_1^+,\si_2^+$ be the strength of the outgoing waves. Note that $y_1(t)$ crosses the interaction point only if $\si_1^+>0$. Then $v_1(t_0+,y(t_0+))\in[v_1^l,\hat{v}_1^m]$ and $v_1(t_0-,y_1(t_0-))=v_1^r$. Hence, we have
	\begin{align*}
		&\abs{v(t_0+,y_1(t_0+))}-\abs{v(t_0-,y_1(t_0-))}\\
		&\leq \abs{v^l_1-v^r_1}+\abs{\si_1^+}\\
		&\leq C_2(\abs{\si^\p}^3+\abs{\si^{\p\p}}^3)+C_2\abs{\si^\p\si^{\p\p}}(\abs{\si^\p}+\abs{\si^{\p\p}})\\
		&\leq C_3(\abs{\si^\p}+\abs{\si^{\p\p}})^3.
	\end{align*}

\end{proof}
An analogous result also holds for second family minimal characteristic $y_2$ as well. 
	\begin{lemma}\label{lemma:change-vi-1}
			Let $t_n>0$ be such that $v_1(t_n+,y_1(t_n+))\neq v_1(t_n-,y_1(t_n-))$. Then, one of the following is true
		\begin{enumerate}
			\item When $y(t_n)\neq x_j$ for all $j\in\mathbb{Z}$,
			\begin{enumerate}
				
			\item $(t_n\pm,y_1(t_n\pm))$ are not discontinuity points of $u_\nu(t_n\pm,\cdot)$ respectively, then 
			\begin{equation}\label{change-v_1-t_n-1}
				\abs{v_1(t_n+,y_1(t_n+))}-\abs{v_1(t_n-,y_1(t_n-))}\leq C_3\tau_\nu \omega_{2,n}.
			\end{equation}
			\item $(t_n-,y_1(t_n-))$ is the end point of a 1-wave $\si_1^-$ in $(t_n-\epsilon,t_n)$ for some small $\epsilon>0$, then 
			\begin{equation}\label{change-v_1-t_n-2}
				\abs{v_1(t_n+,y_1(t_n+))}-\abs{v_1(t_n-,y_1(t_n-))}\leq C_3\tau_\nu \omega_{2,n}+C_3\abs{\si_1^-}\tau_\nu \omega_{2,n}.
			\end{equation}
		   \item $(t_n-,y_1(t_n-))$ is the end point of a 2-wave $\si_2^-$ in $(t_n-\epsilon,t_n)$ for some small $\epsilon>0$, then 
		   \begin{align}
		   	&\abs{v_1(t_n+,y_1(t_n+))}-\abs{v_1(t_n-,y_1(t_n-))}\nonumber\\
		   	&\leq C_3\left(\tau_\nu \omega_{2,n}+\tau_\nu \omega_{2,n}\abs{\si_2^-}+\abs{\si_2^-}^3\right). \label{change-v_1-t_n-3}
		   \end{align}

	\end{enumerate}
            \item When $y(t_n)= x_j$ for some $j\in\mathbb{Z}$,
            \begin{enumerate}
            	\item $(t_n-,y_1(t_n-))$ is the end point of a 1-wave $\si_1^-$ in $(t_n-\epsilon,t_n)$ for some small $\epsilon>0$, then 
            	\begin{align}
            		&\abs{v_1(t_n+,y_1(t_n+))}-\abs{v_1(t_n-,y_1(t_n-))}\nonumber\\
            		&\leq C_3\tau_\nu \omega_{2,n}+C_3\tau_\nu \omega_{2,n}\abs{\si^-_1}. \label{change-v_1-t_n-4}
            	\end{align}
            	\item $(t_n-,y_1(t_n-))$ is the end point of a 2-wave $\si_2^-$ in $(t_n-\epsilon,t_n)$ for some small $\epsilon>0$, then 
            	\begin{align}
            		&\abs{v_1(t_n+,y_1(t_n+))}-\abs{v_1(t_n-,y_1(t_n-))}\nonumber\\
            		&\leq C_3\tau_\nu \omega_{2,n}+C_3\abs{\si^-_2}^3+C_3\tau_\nu \omega_{2,n}\left(\abs{\si^-_2}+\e_\nu\omega_{1,j}\right). \label{change-v_1-t_n-5}
            	\end{align}

            \end{enumerate}
		\end{enumerate}
			\end{lemma}
		\begin{proof}
		We set up a general notations for all the cases as follows. Let $v^l_-$ and $v^r_-$ be the states (corresponding to $u^l_-$ and $u_-^r$) before $t_n$ and $v_+^l,v_+^r$ be the states (corresponding to $u^l_+$ and $u_+^r$)  at $t=t_n+$. Let $\si_1^+,\si_2^+$ be the 1- and 2-wave arsing from the Riemann problem solution to $(v_+^l,v_+^r)$ with middle state $v^m_+$. From \eqref{defn:time-step} we have 
		\begin{equation*}
			u_+^l=u_-^l+\tau_\nu g_\nu(t_n-,y(t_n)-,u_-^l)\mbox{ and }u_+^r=u_-^r+\tau_\nu g_\nu(t_n-,y(t_n)+,u_-^r).
		\end{equation*}
	    In all the cases, we must see that $\si_1^+>0$ and $v_1(t_n+,y_1(t_n+))=v_+^l$. Note that \eqref{change-v_1-t_n-1} follows from \eqref{defn:time-step} and bound on $g_\nu$. To show \eqref{change-v_1-t_n-2} we note in this case, $v_1(t_n-,y_1(t_n-))=v_{1,-}^l$ and $\si_1^->0$. Hence, we get
	   \begin{align*}
	   	\abs{v_1(t_n+,y_1(t_n+))}-	\abs{v_1(t_n-,y_1(t_n-))}&\leq \abs{v^l_{1,+}-v^l_{1,-}}\\
	   	&\leq C_3\tau_\nu \omega_{2,n}+C_3\abs{\si_1^-}\tau_\nu \omega_{2,n}.
	   \end{align*}
       To show \eqref{change-v_1-t_n-3}, we see that $v_1(t_n-,y_1(t_n-))=v_{1,-}^r$. Now, from \eqref{defn:time-step} and the definition of $S_2$ curve we get
       \begin{align*}
       		\abs{v_1(t_n+,y_1(t_n+))}-	\abs{v_1(t_n-,y_1(t_n-))}&\leq  \abs{v^l_{1,+}-v^l_{1,-}}+\abs{\si^1_+}+\abs{v^r_{1,-}-v^l_{1,-}}\\
       	&\leq C_3\tau_\nu \omega_{2,n}+C_3\abs{\si_2^-}\tau_\nu \omega_{2,n}+C_3\abs{\si_2^-}^3.
       \end{align*}
		Next we prove \eqref{change-v_1-t_n-4}. By a similar argument as before we obtain
		\begin{align*}
			&\abs{v_1(t_n+,y_1(t_n+))}-	\abs{v_1(t_n-,y_1(t_n-))}\\
			&\leq \abs{v^l_{1,+}-v^l_{1,-}}\\
			&\leq C_3\tau_\nu \omega_{2,n}+C_3\tau_\nu \omega_{2,n}\abs{\si_1^-}.
		\end{align*}
      To estimate \eqref{change-v_1-t_n-5} note that $v_1(t_n-,y_1(t_n-))=v_{1,-}^r$ and $\abs{\si_1^+}\leq C_3\tau_\nu \omega_{2,n}(\abs{\si_2^-}+\epsilon_\nu\omega_{1,j})$. Therefore,
      \begin{align*}
      	&\abs{v_1(t_n+,y_1(t_n+))}-	\abs{v_1(t_n-,y_1(t_n-))}\\
      	&\leq  \abs{v^l_{1,+}-v^l_{1,-}}+\abs{\si^1_+}+\abs{v^r_{1,-}-v^l_{1,-}}\\
      	&\leq C_3\tau_\nu \omega_{2,n}+C_3\tau_\nu \omega_{2,n}(\abs{\si_2^-}+\epsilon_\nu\omega_{1,j})+C_3\abs{\si_2^-}^3.
      \end{align*}
  This completes the proof of Lemma \ref{lemma:change-vi-1}.
		\end{proof}

	
	We define
	\begin{align}
		\tilde{Q}_1(t)&:=\sum\limits_{x_\al<y_1(t)}\abs{\si_{2,\al}},\\
		\tilde{Q}_2(t)&:=\sum\limits_{x_\al>y_2(t)}\abs{\si_{1,\al}}.
	\end{align}
	\begin{lemma}\label{lemma-tQ}
		Let $i,j\in\{1,2\}$ with $i\neq j$ and $t>0$. The following cases hold true.
		\begin{enumerate}
			\item There is no interaction and $y_i(t)$ does not cross any wave, then $\De \tilde{Q}_i(t)=0$.
			\item There is no interaction and $y_i(t)$ crosses a $j$-wave $\si_j$, then $\De \tilde{Q}_i(t)=-\abs{\si_j}$.
			

			\item There is an interaction between waves $\si^\p,\si^{\p\p}$ (at least one of them is $\leq -\epsilon_\nu$) and $y_i(t)$ does not cross any wave then
			\begin{equation*}\label{estimate-tQ-3}
				\De \tilde{Q}_i(t)\leq C_1\abs{\si^\p\si^{\p\p}}(\abs{\si^\p}+\abs{\si^{\p\p}}).
			\end{equation*}

			\item There is an interaction between waves $\si^\p,\si^{\p\p}$ and $y_i(t)$ crosses a $j$-wave then
			\begin{equation}\label{estimate-tQ-4}
				\De \tilde{Q}_i(t)\leq C_1\abs{\si^\p\si^{\p\p}}(\abs{\si^\p}+\abs{\si^{\p\p}})-\abs{\si_j}.
			\end{equation}

			\item There is an interaction between $j$-waves $\si^\p,\si^{\p\p}$ and $y_i(t)$ crosses the interaction point then
			\begin{equation}\label{estimate-tQ-5}
				\De \tilde{Q}_i(t)\leq -\abs{\si^\p}-\abs{\si^{\p\p}}.
			\end{equation}

	       \item At $t=t_n$, if $y_i(t)$ does not cross any wave at $(t_n,y(t_n))$,
	       \begin{equation}\label{estimate-tQ-6}
	       	\De\tilde{Q}(t_n)\leq C_3\tau_\nu\omega_{2,n} V(t_n-)+C_3 \tau_\nu\omega_{2,n}\sum\limits_{k}\e_\nu\omega_{1,k}.
	       \end{equation}
			
			\item At $t=t_n$, if $y_i(t)$ crosses a 2-wave at $(t_n,y(t_n))$,
			\begin{equation}\label{estimate-tQ-7}
				\De\tilde{Q}(t_n)\leq -\abs{\si_2^-}+C_3\tau_\nu\omega_{2,n} V(t_n-)+C_3 \tau_\nu\omega_{2,n}\sum\limits_{k}\e_\nu\omega_{1,k}.
			\end{equation}

		\end{enumerate}
	\end{lemma}
\begin{proof}
	We will prove the last two cases and rest of the cases follow from definition of $\tilde{Q}_i$. First we consider $i=1$. To show \eqref{estimate-tQ-6} note that
	\begin{align*}
			\De\tilde{Q}(t_n)&\leq \sum\limits_{x_\al<y_1(t)}\abs{\si_{2,\al}^--\si_{2,\al}^+}\\
			&\leq C_2\tau_\nu\omega_{2,n}\sum\limits_{x_\al<y_1(t_n)}\abs{\si_{2,\al}^-}+C_2 \tau_\nu\omega_{2,n}\sum\limits_{x_\al<y_1(t_n)}\epsilon_\nu \omega_{1,k}\\
			&\leq C_2\tau_\nu\omega_{2,n}V(t_n-)+C_2 \tau_\nu\omega_{2,n}\sum\limits_{k}\epsilon_\nu \omega_{1,k}.
	\end{align*}
The case $i=2$ follows in the same way. Similarly, we can prove the estimates \eqref{estimate-tQ-7}.
\end{proof}
	
	Let $\mathcal{A}$ be the set of all approaching waves. We consider the total strength of waves and the interaction potential as follows.
	\begin{equation}
		V(v^\e):=\sum\limits_{i,\al}\abs{\si_{i,\al}}\mbox{ and }Q(v^\e):=K\sum\limits_{(\si_{i,\al},\si_{j,\B})\in\mathcal{A}}\abs{\si_{i,\al}\si_{j,\B}}.
	\end{equation}
To take into account the additional changes in total variation due to forcing term we define, 
\begin{equation}
	\hat{Q}(t,v^\e):=Q(v^\e)+R(V(v^\e)+\norm{\omega_{1}^\nu}_{L^1})\tau_\nu\sum\limits_{t_k\geq t}\omega_{2,k}.
\end{equation}
	
	\begin{proposition}\label{proposition-Q}
		Fix $M_1>0$. Suppose that $v^\e$ is well defined for $t\in[0,t_*)$ for some $t_*<T^*$. Assume that 
		\begin{align}
			TV(v^\e(t))&<M_1,\\
      \norm{\omega_2^\nu}_{L^1(0,T^*)}&\leq 1/(4C_3),\\
			\norm{v^\e(t_*-)}_{L^\f}&\leq 1/(4C_3M_1).
		\end{align} 
	Furthermore, we choose $K$ and $R$ in definition of $\hat{Q}$ such that 
		\begin{align}
			K&\geq 8RC_1\norm{v(t_*-)}_{L^\f}\norm{\omega_2^\nu}_{L^1(t_*,T^*)},\\
			R&\geq 4KC_3(V(t_*-)+2\norm{\omega_1^\nu}_{L^1}).
		\end{align}
	Then the following holds true.
		\begin{enumerate}
			\item  If two waves $\si^\p,\si^{\p\p}$ (of $i,j$-th family for $i,j\in\{1,2\}$) interact at $t_*\neq t_n$ for all $n\geq1$, then $v^\e$ can be defined beyond $t$ and 
			\begin{equation}
				\De \hat{Q}(v^\e(t_*))\leq -\frac{K}{4}\abs{\si^\p\si^{\p\p}}.
			\end{equation}
			\item For $t_*=t_n$ for some $n\in\N$, 
			\begin{equation}
				\De \hat{Q}(v^\e(t_*))\leq -\frac{R}{2}(V(t_n-)+\norm{\omega_1^\nu}_{L^1})\tau_\nu\omega_{2,n}.
			\end{equation}
		\end{enumerate} 
		
	\end{proposition}
	\begin{proof}
		We first consider $t_*\in (t_n,t_{n+1})$. By our construction two waves $\si^{\p},\si^{\p\p}$ meet at some point $\hat{x}$. Then from Lemma \ref{lemma-1}, it follows, 
		\begin{equation}
			\De V(t_*)\leq C_1\abs{\si^\p\si^{\p\p}}(\abs{\si^\p}+\abs{\si^{\p\p}}).
		\end{equation}
	   Then we get
	   \begin{align*}
	   	\De Q(v^\e(t_*))&\leq -K\abs{\si^\p\si^{\p\p}}+KC_1TV(v^\e(t_*-))\abs{\si^\p\si^{\p\p}}(\abs{\si^\p}+\abs{\si^{\p\p}})\\
	   	&\leq \abs{\si^\p\si^{\p\p}}\left(-K+2KC_3M_1\norm{v^\e(t_*-)}_{L^\f}\right)\\
	   	&\leq -\frac{K}{2}\abs{\si^\p\si^{\p\p}},
	   \end{align*}
     provided $\norm{v^\e(t_*-)}_{L^\f}\leq1/(4M_1C_3)$. Subsequently, 
     \begin{align*}
     	\De \hat{Q}(t_*,v^\e(t_*))&\leq -\frac{K}{2}\abs{\si^\p\si^{\p\p}}+R\abs{\De V(v^\e(t_*))}\tau_\nu\sum\limits_{t_k\geq t_*}\omega_{2,k}\\
     	&\leq -\frac{K}{2}\abs{\si^\p\si^{\p\p}}+RC_1\abs{\si^\p\si^{\p\p}}(\abs{\si^\p}+\abs{\si^{\p\p}})\tau_\nu\sum\limits_{t_k\geq t_*}\omega_{2,k}\\
     	&\leq -\frac{K}{4}\abs{\si^\p\si^{\p\p}},
     \end{align*}
 for $K\geq 8RC_1\norm{v(t_*-)}_{L^\f}\norm{\omega_2^\nu}_{L^1(t_*,T^*)}$. Now, we consider the change in $V, Q$ and $\hat{Q}$ when $t_*=t_n$. From definition, we have
 \begin{align}
 	\De V(t_n)&\leq C_3 V(t_n-)\tau_\nu\omega_{2,n}+C_3\norm{\omega_1^\nu}_{L^1(\R)}\tau_\nu\omega_{2,n},\\
 	\De Q(v^\e(t_n))&\leq  KC_3V(t_n-)\tau_\nu\omega_{2,n}\left(V(t_n-)+\norm{\omega_1^\nu}_{L^1}\right)+KC_3\left(\tau_\nu\omega_{2,n}\norm{\omega_{1}^\nu}_{L^1}\right)^2.
 \end{align}
Then we get
\begin{align*}
	&\De \hat{Q}(t_n,v^\e(t_n))\\
	&\leq KC_3V(t_n-)\tau_\nu\omega_{2,n}[V(t_n-)+\norm{\omega_1^\nu}_{L^1}]+KC_3\left(\tau_\nu\omega_{2,n}\norm{\omega_{1}^\nu}_{L^1}\right)^2\\
	&+R\abs{\De V(t_n)}\tau_\nu\sum\limits_{k\geq n+1}\omega_{2,k}-R(V(t_n-)+\norm{\omega_1^\nu}_{L^1})\tau_\nu\omega_{2,n}\\
	&\leq KC_3V(t_n-)\tau_\nu\omega_{2,n}[V(t_n-)+\norm{\omega_1^\nu}_{L^1}]+KC_3\left(\tau_\nu\omega_{2,n}\norm{\omega_{1}^\nu}_{L^1}\right)^2\\
	&+RC_3 V(t_n-)\tau_\nu\omega_{2,n}\tau_\nu\sum\limits_{k\geq n+1}\omega_{2,k}+RC_3\norm{\omega_1^\nu}_{L^1(\R)}\tau_\nu\omega_{2,n}\tau_\nu\sum\limits_{k\geq n+1}\omega_{2,k}\\
	&-R(V(t_n-)+\norm{\omega_1^\nu}_{L^1})\tau_\nu\omega_{2,n}\\
	&=RV(t_n-)\tau_\nu\omega_{2,n}\left[\frac{KC_3}{R} \left(V(t_n-)+\norm{\omega_1^\nu}_{L^1}\right)+C_3\norm{\omega_2^\nu}_{L^1(t_n,T^*)}-1\right]\\
	&+R\norm{\omega_1^\nu}_{L^1}\tau_\nu\omega_{2,n}\left[\frac{KC_3}{R}\tau_\nu\omega_{2,n}\norm{\omega_1^\nu}_{L^1}+C_3\norm{\omega_2^\nu}_{L^1(t_n,T^*)}-1\right]\\
	&\leq -\frac{R}{2}(V(t_n-)+\norm{\omega_1^\nu}_{L^1})\tau_\nu\omega_{2,n}
\end{align*}
provided $C_3\norm{\omega_2^\nu}_{L^1(t_n,T^*)}\leq 1/4$ and
\begin{equation}
	\frac{KC_3}{R}(V(t_n-)+2\norm{\omega_1^\nu}_{L^1})\leq \frac{1}{4}.
\end{equation}
This completes the proof of Proposition \ref{proposition-Q}. 
 \end{proof}
We define
	\begin{align}
		\Upsilon(v^\e(t))&=V(v^\e(t))+\hat{Q}(v^\e(t)),\\
		\Theta(v^\e(t))&=\left(\abs{v_i^\e(t,y_i(t))}+\norm{\bar{v}^\e}_{L^\f}+A_0W(t)\right)e^{\tilde{A}\tilde{Q}_i(t)+A\hat{Q}(v^\e(t))},
	\end{align}
	for some $K,A_0,\tilde{A},A>0$ which will be chosen later and $W(t)$ is defined as follows
	\begin{equation}
		W(t):=\tau_\nu\sum\limits_{t_n\geq t}\omega_{2,n}.
	\end{equation}
	
	\begin{lemma}\label{lemma-decreasing}
		Fix $M_1,M_2>0$. Consider that an initial data $\bar{v}^\e$ is satisfying $\norm{\bar{v}^\e}_{L^\f}<\eta$ and the force $g$ is verifying $\norm{\omega_1}_{L^1(\R)}+\norm{\omega_2}_{L^1(0,T^*)}<\eta$. We assume that the corresponding approximate solution $v^\e$ is defined till time $t>0$. There exists $\eta>0$ and $\tilde{A},A,A_0,K,R$ such that if $TV(v^\e(t-))<M_1$ and $\norm{v^\e(t-)}_{L^\f}<M_2\norm{\bar{v}^\e}_{L^\f}$, then 
		\begin{equation}
			\De \Upsilon(v^\e(t))\leq 0,\mbox{ and }\De\Theta_i(v^\e(t))\leq 0\mbox{ for }i=1,2.
		\end{equation}
	\end{lemma}
	\begin{proof}
			Consider a situation when $\si^\p,\si^{\p\p}$ interact at some point $(x_*,t_*)$ for $t_*\in(t_n,t_{n+1})$.
		\begin{align}
			\De \Upsilon(t_*,v^\e(t_*))&\leq C_3\abs{\si^\p\si^{\p\p}}(\abs{\si^\p}+\abs{\si^{\p\p}})-\frac{K}{4}\abs{\si^\p\si^{\p\p}}\\
			&\leq 0
		\end{align}
			provided $K\geq 8C_3\norm{v(t_*-)}_{L^\f}$. At $t=t_n$ we obtain
			\begin{align*}
					\De \Upsilon(t_n,v^\e(t_n))
					&\leq \De V(t_n)+\De \hat{Q}(t_,v^\e(t_n))\\
					&\leq C_3 V(t_n-)\tau_\nu\omega_{2,n}+C_3\norm{\omega_1^\nu}_{L^1(\R)}\tau_\nu\omega_{2,n}\\
					&-\frac{R}{2}\left(V(t_n-)+\norm{\omega_1^\nu}_{L^1}\right)\tau_\nu\omega_{2,n}\\
					&\leq 0,
			\end{align*}
		provided $R\geq 2C_3$. 
		
		Now we focus on change of $\Theta_i$. 
	\begin{enumerate}
		\item If there is no interaction and $y_i(t)$ does not cross any wave, then $\De \tilde{Q}_i(t)=0$, $\De \Theta_i=0$. 
		
		\item Consider the situation there is no interaction and $y_i(t)$ crosses a $j$-wave $\si_j$, then $\De \tilde{Q}_i(t)=-\abs{\si_j}$. From Lemma \ref{lemma:change-vi}, we can see that $\De v_i^\e(t,y_i(t))\leq C_3\norm{v(t-)}^2_{L^\f}\abs{\si_j}$. This implies $\De \Theta_i\leq 0$ provided $\tilde{A}\geq C_3\norm{v(t-)}_{L^\f}^2$. 
		
%
		
		\item Now we consider the situation when there is an interaction between waves $\si^\p,\si^{\p\p}$ and $y_i(t)$ does not cross any wave. From \eqref{estimate-tQ-3} and Proposition \ref{proposition-Q} we get then
		\begin{equation*}
			\De \tilde{Q}_i(t)\leq C_1\abs{\si^\p\si^{\p\p}}(\abs{\si^\p}+\abs{\si^{\p\p}})\mbox{ and }\De \hat{Q}\leq -\frac{K}{4}\abs{\si^\p\si^{\p\p}}.
		\end{equation*} 
	  Since $\De v_i^\e(t,y_i(t))=0$, we get $\De \Theta_i\leq 0$ provided $KA\geq 8C_1\norm{v(t-)}_{L^\f}\tilde{A}$. 
	  \item Now if there is an interaction between waves $\si^\p,\si^{\p\p}$ and $y_i(t)$ crosses a $j$-wave then
	  \begin{align*}
	  	\De \tilde{Q}_i(t) &\leq C_1\abs{\si^\p\si^{\p\p}}(\abs{\si^\p}+\abs{\si^{\p\p}})-\abs{\si_j},\\
	  	\De v_i^\e(t,y_i(t))&\leq C_3\norm{v(t-)}^2_{L^\f}\abs{\si_j},\\
	  	\De \hat{Q}&\leq -\frac{K}{4}\abs{\si^\p\si^{\p\p}}.
	  \end{align*} 
   Hence, we get $\De \Theta_i\leq 0$ provided $KA\geq 8C_1\norm{v(t-)}_{L^\f}\tilde{A}$ and $\tilde{A}\geq C_3\norm{v(t-)}_{L^\f}^2$. 
   
   \item We consider that there is an interaction between waves $\si^\p,\si^{\p\p}$ and $y_i(t)$ crosses the interaction point then
   \begin{equation*}
   		\De v_i^\e(t,y_i(t))\leq 4C_3\norm{v(t-)}^2_{L^\f}(\abs{\si^\p}+\abs{\si^{\p\p}})\mbox{ and } \De \tilde{Q}_i(t)\leq -\abs{\si^\p}-\abs{\si^{\p\p}}.
   \end{equation*} 
  Hence, we get $\De \Theta_i\leq 0$ provided $\tilde{A}\geq 4C_3\norm{v(t-)}_{L^\f}^2$.
	\item Now we calculate the change at $t=t_n$, by using the estimates \eqref{estimate-tQ-6} and Lemma \ref{lemma:change-vi} we obtain
	\begin{align*}
			\De\tilde{Q}(t_n)&\leq C_3\tau_\nu\omega_{2,n} (V(t_n-)+\norm{\omega_1^\nu}_{L^1}),\\
			\De v(t_n,y_i(t_n))&\leq C_3\tau_\nu\omega_{2,n}.
	\end{align*}
 Furthermore, from the definition of $W$ and Proposition \ref{proposition-Q} we have
 \begin{equation*}
 	\De W(t_n)=-\tau_\nu\omega_\nu(t_n) \mbox{ and }\De \hat{Q}(t_n)\leq -\frac{R}{2}(V(t_n-)+\norm{\omega_1^\nu}_{L^1})\tau_\nu\omega_{2,n}.
 \end{equation*}
 Therefore, we get $\De \Theta_i\leq 0$ provided $A_0\geq C_3$ and $R\geq 2C_3$.
\end{enumerate}
From all the above cases we need the following conditions to be held,
\begin{align*}
	R&\geq 4KC_3(V(t_n-)+\tau_\nu\omega_{2,n}\norm{\omega_1^\nu}_{L^1}),\\
	K&\geq 8RC_1\norm{v(t_*-)}_{L^\f}\norm{\omega_2}_{L^1(t_*,T^*)},\\
	K&\geq 8C_3\norm{v(t_*-)}_{L^\f},\\
	KA&\geq 8C_1\norm{v(t-)}_{L^\f}\tilde{A},\\
	\tilde{A}&\geq 4C_3\norm{v(t-)}_{L^\f}^2,\\
	R&\geq 2C_3,\,\, A_0\geq C_3.
\end{align*} 
We set 
\begin{equation*}
	R=4C_3,\,\,A_0=2C_3,\,\, \tilde{A}=4C_3M_1^2\norm{\bar{v}}_{L^\f}^2,\,\,K=8C_3M_1\norm{\bar{v}}_{L^\f},\,\,A=4C_3M_1^2\norm{\bar{v}}_{L^\f}^2.
\end{equation*}
They satisfy the required inequalities provided 
\begin{align*}
	8C_3\norm{\omega_2}_{L^1(0,T^*)}&\leq1,\\
8C_3M_1\norm{\bar{v}}_{L^\f}\left[M_1+\norm{\omega_2}_{L^1(0,T^*)}\norm{\omega_1}_{L^1(\R)}\right]&\leq 1/2. 
\end{align*}
This completes the proof of Lemma \ref{lemma-decreasing}. 
\end{proof}
\begin{remark}
	\begin{equation*}
		R=4C_3,\,\,A_0=2C_3,\,\, \tilde{A}=4C_3M_1^2\norm{\bar{v}}_{L^\f}^2,\,\,K=8C_3M_1\norm{\bar{v}}_{L^\f},\,\,A=4C_3M_1^2\norm{\bar{v}}_{L^\f}^2.
	\end{equation*}
\end{remark}
	\begin{proposition}
		Let $M_1>0$ as in Lemma \ref{lemma-decreasing}. There exists $B_1,B_2>0$ and small constants $\eta_1>0,\e_0>0$ such that for $\bar{K}<M_1/B_1$, $\eta<\eta_1,\e<\e_0$ the following holds true. If the initial data $\bar{v}^\e\in \mathcal{D}(\eta,\bar{K})$, $\norm{\omega_1}_{L^1(\R)}+\norm{\omega_2}_{L^1(0,T^*)}<\eta$ and the approximate solution $v^\e$ is defined in $[0,T]\times \R$ for some $0<T<T^*$, then the following holds true,
		\begin{equation}
			TV(v^\e(t))\leq B_1\bar{K}\mbox{ and }\norm{v^\e(t)}_{L^\f}\leq B_2\eta.
		\end{equation}
	\end{proposition}
	\begin{proof}
		To prove the $L^\f$ estimate we use Lemma \ref{lemma-decreasing} and get
		\begin{align*}
			\abs{v^\e(t,x)}&\leq \Theta_i(v^\e(t))\\
			&\leq \Theta_i(v^\e(0))\\
			&\leq (2+2C_3)\eta e^{4C_3M_1^2\eta^2(\tilde{Q}(0)+\hat{Q}(0))}\\
			&\leq (2+2C_3)\eta e^{4C_3M_1^2\eta^2\bar{K}(1+8C_3M_1\eta \bar{K}+4C_3\eta)}\\
			&\leq B_2\eta.
		\end{align*}
		Similarly, again by using Lemma \ref{lemma-decreasing}, we obtain
		\begin{align*}
			TV(v^\e(t))&\leq \Upsilon(v^\e(t))\\
			&\leq \Upsilon(v^\e(0))\\
			&\leq \bar{K}(1+8C_3M_1\eta \bar{K}+4C_3\eta)\\
			&\leq 2\bar{K}.
		\end{align*}
	\end{proof}
\begin{proposition}[Estimation of rarefaction strengths, \cite{Bressan-book}]\label{prop:rarefaction}
	Let $v^\e$ be the $\e$-approximate front tracking solution. Then there exists $C>0$ such that the following holds
	\begin{equation}
		(\si)_+\leq C\e\mbox{ for any front }\si.
	\end{equation}
\end{proposition}
Proof of Proposition \ref{prop:rarefaction} follows from a similar argument as in \cite[Section 7.3, page 138-139]{Bressan-book}. Here we need to use the functional $\hat{Q}$ instead of $Q$. We omit the proof.
	\begin{proposition}
		The approximate solution $v^\e$ constructed in section for $\bar{v}^\e\in\mathcal{D}(\eta,\bar{K})$ is well-defined for all $t>0$ and $v^\e(t)\in\mathcal{D}(B_2\eta,B_1\bar{K})$. Furthermore, there are finitely many shock fronts with strength $\si<-\sqrt{\e}$ and the set of interaction points has no limit points.
	\end{proposition}
	
	\begin{proposition}\label{prop-construction-1}
		Fix $\bar{K}>0$. There exists $\eta,M>0$ such that for any $\bar{v}\in\mathcal{D}(\eta,\bar{K})$ and for sufficiently small $\e>0$ the Cauchy problem \eqref{eqn-main}--\eqref{eqn-data} admits $\e$-approximate solution $v^\e$ satisfies
		\begin{equation}
			\norm{v^\e(t)}_{L^\f}\leq M\norm{\bar{v}}_{L^\f}.
		\end{equation}
	\end{proposition}
	
	\begin{theorem}\label{thm-construction-1}
		Fix $\bar{K}>0$. Then there exist $\eta>0,\mathcal{M}>0$ such that for every $\bar{v}\in\mathcal{D}(\eta,\bar{K})$ and for every sufficiently small $\e>0$, the $\e$-approximate solution $v^\e$ to the Cauchy problem \eqref{eqn-main}--\eqref{eqn-data} constructed in Proposition \ref{prop-construction-1} satisfies the following for all $t>0$
		\begin{equation}\label{decay-estimate}
			TV^+(v_i^\e(t);[a,b])\leq \frac{b-a}{c_1t}+\mathcal{M}\left(\norm{\bar{v}}_{L^\f}+\norm{\omega_2}_{L^1(0,T^*)}\right)TV(\bar{v};[a-\hat{\la}t,b+\hat{\la}t]),
		\end{equation}
		where $\hat{\la}>0$ is as in \eqref{def:max-speed} and some constant $c_1>0$ depending on $\bar{K}$.
	\end{theorem}
	
	\begin{proof}
		We check the decay of positive waves. To fix the ideas let us consider the case for $i=1$. Set $I=[a,b)$ for some $a<b$. Let $a(t)=y_1(t,a)$ and $b(t)=y_1(t,b)$. Define $I(t):=[a(t),b(t))$, $\tilde{I}(t)=(a-\hat{\la}(T-t),b+\hat{\la}(T-t))$ and $z(t)=b(t)-a(t)$. We define
		\begin{equation}
			M(t):=\sum\limits_{x_\al\in I(t)}(\la_1(u(t,x_\al+))-\la_1(u(t,x_\al-)))\mbox{ and }
			K(t):=\sum\limits_{\B\in \tilde{I}(t)}\abs{\si_{2,\B}}.
		\end{equation} 
		Observe that
		\begin{equation}
			\dot{z}(t)\geq M(t)-C_2K(t).
		\end{equation}
		We set
		\begin{equation}
			\phi(t,x):=\left\{\begin{array}{cl}
				0&\mbox{ if }x<a(t),\\
				\frac{x-a(t)}{z(t)}&\mbox{ if }x\in[a(t),b(t)),\\
				1&\mbox{ if }x\geq b(t).
			\end{array}\right.
		\end{equation}
		and define
		\begin{equation}
			\Phi(t):=\sum\limits_{x_\al\in I(t)}\phi(t,x_\al(t))\abs{\si_{2,\al}}.
		\end{equation}
		For a non-interaction point $t$ we get
		\begin{equation}\label{inequality-Phi-K-z}
			\dot{\Phi}(t)=\sum\limits_{x_\al\in \tilde{I}(t)}\left(\frac{\dot{x}_\al-\dot{a}(t)}{z}-\frac{(x_\al-a)\dot{z}}{z^2}\right)\abs{\si_{2,\al}}\geq \frac{c_0}{2}\frac{K(t)}{z(t)}\geq0.
		\end{equation}
		At an interaction time $t_*\in(t_n,t_{n+1})$, we obtain
		\begin{equation}
			\De M(t_*)\leq C_3\abs{\De \hat{Q}(t)}.
		\end{equation}
	    We have the same estimate at $t=t_n$. Then we get
		\begin{equation}
			M(T)-M(t)\leq C_3(\hat{Q}(t)-\hat{Q}(T)).
		\end{equation}
		Using \eqref{inequality-Phi-K-z}, we have for $C_4\geq \max\left\{\frac{2C_2}{c_0},C_3\right\}$,
		\begin{align*}
			\dot{z}(t)+C_4\dot{\Phi}(t)z(t)&\geq M(t)-C_2K(t)+\frac{C_4c_0}{2}K(t)\\
			&\geq M(T)-C_4(\hat{Q}(0)-\hat{Q}(T)).
		\end{align*}
	Now we consider the case $M(T)>2C_4(\hat{Q}(0)-\hat{Q}(T))$. Then, we get
	\begin{equation*}
		\dot{z}(t)+C_4\dot{\Phi}(t)z(t)\geq \frac{M(T)}{2}.
	\end{equation*}
   Let $0<s_1<\cdots <s_n<T<s_{n+1}$ be interaction times. Then for any $t\in(s_i,s_{i+1})$ we can write
   \begin{align*}
   	\frac{d}{dt}\left(\mbox{exp}\left\{\int\limits_{s_i}^{t}C_4\dot{\Phi}(s)ds\right\}z(t)\right)&\geq \mbox{exp}\left\{\int\limits_{s_i}^{t}C_4\dot{\Phi}(s)ds\right\}\frac{M(T)}{2}\\
   	&\geq \frac{M(T)}{2} \quad\quad\mbox{ since }\dot{\Phi}\geq0.
   \end{align*}
 Subsequently,
 \begin{equation*}
 	\mbox{exp}\left\{\int\limits_{s_i}^{s_{i+1}}C_4\dot{\Phi}(s)ds\right\}z(s_{i+1})-z(s_i)\geq \frac{M(T)}{2}(s_{i+1}-s_i).
 \end{equation*}
Since $\Phi(t)\geq0$, we get
 \begin{equation*}
	e^{C_4 \Phi(s_{i+1}-)}z(s_{i+1})-e^{C_4\Phi(s_i+)}z(s_i)\geq \frac{M(T)}{2}(s_{i+1}-s_i).
\end{equation*}
Then we have
\begin{align*}
	e^{C_4 \Phi(T)}z(T)+\sum\limits_{i=1}^{n}z(s_i)\left(e^{C_4 \Phi(s_{i}-)}-e^{C_4\Phi(s_i+)}\right)-z(0)e^{C_4\Phi(0+)}\geq \frac{M(T)T}{2}.
\end{align*}
We also note that
\begin{equation}
	C_4\Phi(t\pm )\leq \La:=C_4\left[TV(\bar{v},[a-\hat{\la}T,b+\hat{\la}T])+ \hat{Q}(0)\right]\mbox{ for all }t\in (0,T].
\end{equation}
and $z(t)\leq b-a+2\hat{\la}T$ for all $t\in (0,T)$. Hence,
\begin{equation*}
	(b-a)e^\La+\left(\hat{Q}(0)-\hat{Q}(T)\right)\left(b-a+2\hat{\la}T\right)e^\La\geq \frac{M(T)T}{2},
\end{equation*}
equivalently,
\begin{equation*}
	M(T)\leq 2e^\La\frac{b-a}{T}+\frac{2e^\La}{T}\left(\hat{Q}(0)-\hat{Q}(T)\right)\left(b-a+2\hat{\la}T\right).
\end{equation*}
Therefore, we have the following estimate for $M(T)$,
\begin{equation*}
	M(T)\leq 2e^\La\frac{b-a}{T}+\frac{2e^\La}{T}\left(\hat{Q}(0)-\hat{Q}(T)\right)\left(b-a+2\hat{\la}T\right)+2C_4(\hat{Q}(0)-\hat{Q}(T)).
\end{equation*}
Let $c_1>0$ and $\tilde{\mathcal{M}}>0$ such that $c_1\left(2e^\La+2e^\La\La\right)\leq 1$ and $\tilde{\mathcal{M}}=\left(2\hat{\la}e^\La+2C_4\right)$. Then we have
\begin{equation*}
	M(T)\leq \frac{b-a}{c_1T}+\tilde{\mathcal{M}} \hat{Q}(0).
\end{equation*}
Let $[a_l,b_l)$, $l=1,2,\dots,m$ be disjoint intervals with $b_l\leq a_{l+1}$. Then by a similar argument we prove that
\begin{equation}
	\sum\limits_{l}M_l(T)\leq \sum\limits_{l}\frac{b_l-a_l}{c_1T}+\tilde{\mathcal{M}}\hat{Q}(0).
\end{equation}
Now, given $a<b$ we can choose $a_l, b_l$ such that $(a_l,b_{l+1})$ only contains a first generation 1-shock. Since sum of all 1-shocks with more than 1-generation can be estimated by $\hat{Q}(0)-\hat{Q}(T)$, we can write
\begin{align*}
	\sum\limits_{l}	\sum\limits_{x_\al\in (a_l,b_l),\si_{1,\al}>0}\abs{\si_{1,\al}}&= \sum\limits_{l}M_l(T)+ \sum\limits_{l}\sum\limits_{x_\al\in (a_l,b_l),\si_{1,al}<0}\abs{\si_{1,\al}}\\
	&\leq  \sum\limits_{l}M_l(T)+\hat{Q}(0)-\hat{Q}(T)\\
	&\leq \sum\limits_{l}\frac{b_l-a_l}{c_1T}+(\tilde{\mathcal{M}}+1)\hat{Q}(0)\\
	&\leq \frac{b-a}{c_1T}+\mathcal{M}\hat{Q}(0)\mbox{ where }\mathcal{M}=\tilde{\mathcal{M}}+1.
\end{align*}
This completes the proof decay estimate	\eqref{decay-estimate}.	
\end{proof}	
\begin{remark}
	We note that the constant $c_1$ and $\mathcal{M}$ depend on $TV(\bar{v},[a-\hat{\la}T,b+\hat{\la}T]$. Therefore, for $\bar{v}\in\mathcal{D}({\eta},\bar{K})$ we have $\mathcal{M}=\mathcal{M}(\bar{K})$ and $c_1=c_1(\bar{K})$. Furthermore, as $\bar{K}\rr0$, we have $\mathcal{M}(\bar{K})\rr\f$ and $c_1(\bar{K})\rr0$.
\end{remark}	
	\section{Construction of solutions for $L^\f$ data}\label{section:L-infty}
	In this section, we consider initial data satisfying the following assumption,
	\begin{equation}\label{assumption-v-1}
		\bar{v}\in C^1(\R; B(0,\eta))\mbox{ with }\norm{\frac{d\bar{v}}{dx}}_{L^\f}\leq \mathcal{L}.
	\end{equation}
	For $m\geq0$ we define
	\begin{equation}\label{defn-tm}
		\bigtriangleup_m=\left\{(t,x)\in[0,\f)\times\R; t\in [t_m,t_m+\De t_m]\mbox{ and }x\in[-x_m(t),x_m(t)]\right\},
	\end{equation}
	where $t_m,x_m$ are defined as follows
	\begin{equation}
		t_m=\frac{(2^m-1)L}{2\hat{\la}},\,\, \De t_m=2^{m-1}L/\hat{\la}\mbox{ and }x_m(t)=2^mL-\hat{\la}(t-t_m).
	\end{equation}
	We consider the domain
	\begin{equation}
		\mathcal{D}_m\left(\de,\frac{20\hat{\la}}{c_0}\right):=\left\{v\in L^1_{loc}(\R;B(0,\de)): TV(v;2^{m+1}L)\leq \frac{20\hat{\la}}{c_0}\right\}.
	\end{equation}
	
	\subsection{Construction in 0-trapeziod}
	We choose $L>0$ such that
	\begin{equation}
		TV(\bar{v},2L)\leq 20\hat{\la}/c_0.
	\end{equation}
	
	\begin{proposition}
		There exists $\eta>0,M>0,\mathcal{M}>0$ such that for initial data $\bar{v}\in\mathcal{D}_0(\eta,20\hat{\la}/c_0)$, the Cauchy problem \eqref{eqn-main}--\eqref{eqn-data} has a weak entropy solution $v$ defined for $t\in[0,20\hat{\la}/c_0]$ and it is satisfying 
		\begin{align}
			\norm{v(t)}_{L^\f}&\leq M\norm{\bar{v}}_{L^\f}\\
			TV^+(v_i(t),2(L-\hat{\la}t))&\leq \frac{2}{c_1}\frac{L-\hat{\la}t}{t}\\
			&+\mathcal{M}\left(\norm{\bar{v}}_{L^\f}+\norm{\omega_2}_{L^1(0,T^*)}\right)TV(\bar{v};2L). 
		\end{align}
	\end{proposition}
	
	\subsection{Construction in $m$-trapeziod}

	\begin{proposition}
		There exists $\eta>0,M>0,\mathcal{M}>0$ such that $v(t_m)\in\mathcal{D}_0(K\sqrt{\eta},20\hat{\la}/c_0)$, the Cauchy problem \eqref{eqn-main}--\eqref{eqn-data} with data $v(t_m)$ has a weak entropy solution $v$ defined for $t\in[t_m,t_{m+1}]$, and it is satisfying 
		\begin{align*}
			\norm{v(t)}_{L^\f}&\leq M\norm{v(t_m)}_{L^\f},\\
			TV^+(v_i(t),2(2^mL-\hat{\la}t))&\leq \frac{2}{c_1}\frac{2^mL-\hat{\la}t}{t-t_m}\\
			&+\mathcal{M}\left(\norm{v(t_m)}_{L^\f}+\norm{\omega_2}_{L^1(0,T^*)}\right)TV(v(t_m);2^{m+1}L). 
		\end{align*}
	\end{proposition}
	
	\subsection{Proof of Theorem \ref{theorem-main}}
	
	\begin{description}
		\descitem{(H)}{H-1}If the solution $v$ is defined up to time $t_m$ with an initial data satisfying \eqref{assumption-v-1}, then there exists $K>0$ such that for all $m\geq 1$, $\norm{v(t_m)}_{L^\f}\leq K\sqrt{\eta}$ where $\eta$ is as in \eqref{assumption-v-1}.
	\end{description}
	
	\begin{proposition}\label{prop:Oleinik}
		Suppose that a weak entropy solution to \eqref{eqn-main}--\eqref{eqn-data} $v$ exists with an initial data $\bar{v}$ satisfying \eqref{assumption-v-1}. Assume that \descref{H-1}{(H)} holds true. Then for sufficiently small $\eta>0$ the following holds true. If $\norm{\bar{v}}_{L^\f}\leq \eta$ for all $m\in\N$ we have
		\begin{equation}\label{estimate-TV-1}
			TV(v(t_m);2^{m+1}L)\leq\frac{20\hat{\la}}{c_0}.
		\end{equation}
	\end{proposition}
	\begin{proof} We can estimate
		\begin{align*}
			&TV^+(v_i(t_m); 2^{m+1}L)\\
			&\leq 4TV^+(v_i(t_m); 2^{m-1}L)\\
			&\leq \frac{2^{m+1}L}{c_0(t_m-t_{m-1})}+4\mathcal{M}\left(\norm{v(t_{m-1})}_{L^\f}+\norm{\omega_2}_{L^1(0,T^*)})TV(v(t_{m-1});2^mL\right)\\
			&\leq \frac{8\hat{\la}}{c_0}+4\mathcal{M}\left(\norm{v(t_{m-1})}_{L^\f}+\norm{\omega_2}_{L^1(0,T^*)}\right)TV(v(t_{m-1});2^mL).
		\end{align*}
		Subsequently, we have
		\begin{align*}
			TV(v(t_m);2^{m+1}L)\leq &\frac{16\hat{\la}}{c_0}+8\mathcal{M}(\norm{v(t_{m-1})}_{L^\f}+\norm{\omega_2}_{L^1(0,T^*)})TV(v(t_{m-1});2^mL)\\
			&+2\norm{v(t_m)}_{L^\f}.
		\end{align*}
		By using  \descref{H-1}{(H)} with a sufficiently small $\eta>0$ we get the estimate \eqref{estimate-TV-1}.
	\end{proof}
From the definition of $t_m$, we have
    \begin{equation*}
    	2^{m+1}L=\frac{2^{m+1}\hat{\la}2t_{m}}{2^m-1}\geq 4\hat{\la}t_m.
    \end{equation*}
   By Proposition \ref{prop:Oleinik}, 
   \begin{equation}
   				TV(v(t_m);4\hat{\la}t_m)\leq\frac{20\hat{\la}}{c_0}.
   \end{equation}
For any $t\in(t_m,t_{m+1})$ we get
\begin{equation}
	TV(v(t);4\hat{\la}t_m-2\hat{\la}\De t_m)\leq \frac{20\hat{\la}}{c_0}.
\end{equation}
We calculate
\begin{align*}
	4\hat{\la}t_m-2\hat{\la}\De t_m&=4\hat{\la}t_m-2^{m}L\\
	&=4\hat{\la}t_m-\frac{2^m}{2^m-1}\hat{\la}t_m\geq 2\hat{\la}t_m.
\end{align*}
Note that $t\leq \frac{t_{m+1}}{t_m}t_m\leq \frac{2^{m+1}-1}{2^m-1}t_m\leq 4t_m$. Therefore, we get
\begin{equation}\label{Oleinik-1}
	TV(v(t);\hat{\la}t/2)\leq \frac{20\hat{\la}}{c_0}.
\end{equation}
To complete the proof of Theorem \ref{theorem-main}, we first regularize (with parameter $n$) the initial data. Then we can find a $L>0$ such that $TV(\bar{v},L)\leq \bar{K}$. From the above analysis we obtain a solution $v_n$ through front tracking approximation and it satisfies
\begin{equation}
		TV(v_n(t);\hat{\la}\hat{t}/2)\leq \frac{20\hat{\la}}{c_0}\mbox{ for }t\geq \frac{1}{m},
\end{equation}
for some $m>1$. By Helly's theorem we can obtain that up to a subsequence $v_n\rr v$ in $L^1_{loc}$ for some function $v$ in $t\in (1/m,T^*)$. Now, by Cantor's diagonalization argument we obtain that up to a subsequence $v_n$ converges to some function $v(t)>0$ in $L^1_{loc}$. This completes the proof of Theorem \ref{theorem-main} with the assumption \descref{H-1}{(H)}.

\subsection{Integral estimate}
	\begin{lemma}[\cite{Bi-Col-Mon}]\label{lemma-tech-1}
		If $f$ is as in \eqref{formation-f}, the the following holds true.
		\begin{equation}
			(Dr_2r_2)(0)=[-\al_{22},0]^T\mbox{ and } (Dr_1r_1)(0)=[-\B_{11},0]^T.
		\end{equation}
	\end{lemma}
	
	\begin{lemma}[\cite{Bi-Col-Mon}]\label{lemma-tech-2}
		Suppose $f$ satisfies $\frac{\pa^2f_1}{\pa u_2^2}(0)=\al_{22}\neq0$ (respectively $\frac{\pa^2f_2}{\pa u_1^2}(0)=\B_{11}\neq0$)  and the condition \descref{H-f}{($\mathcal{H}_f$)}. Then we have
		\begin{align}
			[S_2^{\p\p\p}(0,0)-R^{\p\p\p}_2(0,0)]_1&=\frac{1}{2}\frac{\langle(D\la_2r_2)(Dr_2r_2),r_1\rangle}{\la_2-\la_1}\neq 0,\\ 
			\mbox{respectively } [S_1^{\p\p\p}(0,0)-R^{\p\p\p}_1(0,0)]_1&=\frac{1}{2}\frac{\langle(D\la_1r_1)(Dr_1r_1),r_2\rangle}{\la_1-\la_2}\neq 0.
		\end{align}
	\end{lemma}
	
	\begin{lemma}\label{lemma-int-1}
		Let $u$ be an entropy solution constructed in section \ref{sec:construction}, satisfying $\norm{u(t)}_{L^\f}\leq C\sqrt{\eta}$ for an initial data verifying \eqref{smallness-condition-1} and \eqref{assumption-v-1}. If $\frac{\pa^2 f_1}{\pa u_2^2}(0)\neq 0$ (respectively $\frac{\pa^2 f_2}{\pa u_1^2}(0)\neq 0$) then there exists an invariant region for $u_1$ (respectively $u_2$). More precisely, there exists $\mathcal{K}>0$ for all $(t,x)\in \R_+\times\R$ the following holds
		\begin{equation}
			u_1(t,x)\geq -\mathcal{K} \eta, \quad \mbox{ respectively }\quad u_2(t,x)\geq -\mathcal{K}\eta.
		\end{equation}
	\end{lemma}	
	\begin{proof}
		Consider $\mathcal{J}(v_1,v_2)=(u_1,u_2)$. Since $\frac{\pa^2f_1}{\pa u_2^2}(0)\neq 0$ by Lemma \ref{lemma-tech-2},
		\begin{equation}
			[\mathcal{S}^{\p\p\p}_2(v,\si)-\mathcal{R}^{\p\p\p}(v,\si)]_1=[\mathcal{S}^{\p\p\p}_2(v,\si)]_1\neq0,
		\end{equation}
		for $v$ sufficiently small. 
		
		First consider the case $[\mathcal{S}^{\p\p\p}_2(v,\si)]_1\leq 0$. Let $u_*$ be the middle state for a Riemann problem with left and right state $u_-,u_+$ respectively. Then we have
		\begin{equation}
			v_1^\e(u_*)\leq v_1^\e(u_+).
		\end{equation}
	   Hence, $v_1(u^\e(t,x))\leq \eta_1$ for all $t\in (t_n,t_{n+1})$ and $x\in\R$ if $v_1(u^\e(t_n+,x))\leq \eta_1$ for all $x\in\R$. We note that for all $x\in\R$,
	   \begin{equation}
	   	\abs{v_1(u^\e(t_n-,x))-v_1(u^\e(t_n+,x))}\leq C_5\tau_\nu \omega_{2,n}\mbox{ where }C_5=C_3\norm{D\mathcal{J}}_{L^\f}.
	   \end{equation}
    Therefore, we obtain
    \begin{align}
    	v_1(u^\e(t_n-,x))&\leq \eta+C_5\tau_\nu \sum\limits_{k=1}^{n-1}\omega_{2,k}\mbox{ for }n\geq 1,\\
    	v_1(u^\e(t_n+,x))&\leq \eta+C_5\tau_\nu \sum\limits_{k=1}^{n}\omega_{2,k}\mbox{ for }n\geq 1,\\
    	v_1(u^\e(t,x))&\leq \eta+C_5\tau_\nu \sum\limits_{k=1}^{n}\omega_{2,k}\mbox{ for }t\in(t_n,t_{n+1})\mbox{ for all }n\geq1.
    \end{align}
	   Since $\norm{\omega_2^\nu}\leq \eta$, $v_1^\e(t,x)\leq (2+C_5)\eta$. We may assume that 
		\begin{equation}
			\mathcal{J}_1(0,0)=0,\,\frac{\pa\mathcal{J}_1}{\pa v_2}(0,0)=0,\,\frac{\pa\mathcal{J}_1}{\pa v_1}(0,0)=-\mathcal{K}_1\mbox{ for some }\mathcal{K}_1>0.
		\end{equation}
		Then for some $\mathcal{K}_2,\mathcal{K}_3,\mathcal{K}_4$, we have
		\begin{align*}
			u_1^\e(t,x)&=\mathcal{J}_1(v_1^\e(t,x),v_2^\e(t,x))\\
			&=-\mathcal{K}_1v_1^\e(t,x)+\mathcal{K}_2(v_1^\e(t,x))^2+\mathcal{K}_3v_1^\e(t,x)v_2^\e(t,x)+\mathcal{K}_4(v_2^\e(t,x))^2\\
			&\geq -(2+C_5)\mathcal{K}_1\eta-C_5(\abs{\mathcal{K}_2}+\abs{\mathcal{K}_3}+\abs{\mathcal{K}_4})\eta\\
			&=-\mathcal{K}\eta\mbox{ for }\mathcal{K}:=(2+C_5)\mathcal{K}_1+\tilde{C}(\abs{\mathcal{K}_2}+\abs{\mathcal{K}_3}+\abs{\mathcal{K}_4}).
		\end{align*}
		If $[\mathcal{S}^{\p\p\p}_2(v,\si)]_1\geq 0$, we have
		\begin{equation}
			v_1^\e(u_*)\geq v_1^\e(u_+).
		\end{equation}
	 By a similar argument as before, we get
	 \begin{align}
	 	v_1(u^\e(t_n-,x))&\geq -\eta-C_5\tau_\nu \sum\limits_{k=1}^{n-1}\omega_{2,k}\mbox{ for }n\geq 1,\\
	 	v_1(u^\e(t_n+,x))&\geq -\eta-C_5\tau_\nu \sum\limits_{k=1}^{n}\omega_{2,k}\mbox{ for }n\geq 1,\\
	 	v_1(u^\e(t,x))&\geq -\eta-C_5\tau_\nu \sum\limits_{k=1}^{n}\omega_{2,k}\mbox{ for }t\in(t_n,t_{n+1})\mbox{ for all }n\geq1.
	 \end{align}
     Since $\norm{\omega_2^\nu}\leq \eta$, $v_1^\e(t,x)\geq -(2+C_5)\eta$. By a similar argument as before, we obtain $u_1^\e(t,x)\geq -\mathcal{K}\eta$. Similar estimates holds for $u_2$ when $\frac{\pa^2 f_2}{\pa u_1^2}(0)\neq 0$.
	\end{proof}

	\begin{proposition}\label{proposition-int-1}
		Let $v$ be the solution constructed in section \ref{sec:construction} corresponding to an initial data verifying \eqref{smallness-condition-1} and \eqref{assumption-v-1}. If $\frac{\pa^2 f_1}{\pa u_2^2}(0)\neq 0$ and $\frac{\pa^2 f_2}{\pa u_1^2}(0)\neq 0$, then for any interval $I$ of length $l$ and $\bar{t}\geq0$, 
		\begin{equation}
			\abs{\int\limits_{I}v_{i}(\bar{t},x)\,dx}\leq C^\p\eta(l+C^{\p\p}\bar{t}).
		\end{equation}
	\end{proposition}
	\begin{proof}
		From Lemma \ref{lemma-int-1} we get
		\begin{equation}
			\abs{u_1}\leq u_1+2\mathcal{K}\eta\mbox{ and } 	\abs{u_2}\leq u_2+2\mathcal{K}\eta
		\end{equation} 
		Consider a trapezoid $\mathcal{T}$ in $t,x$ plane with (i) lower basis $I_0$ connecting the points $(0,x_1)$ and $(0,x_2)$ (ii) upper basis $I$ connecting the points $(\bar{t},x_1+\theta \bar{t})$ and $(\bar{t},x_2-\theta \bar{t})$ for some $\theta>0$. Then applying divergence theorem we obtain
		\begin{align*}
			&\int\limits_{I}[u_1(\bar{t},x)+u_2(\bar{t},x)]\,dx\\
			&=\int\limits_{I_0}[u_1(0,x)+u_2(0,x)]\,dx+\int\limits_{0}^{\bar{t}}\int\limits_{x_1+\theta{t}}^{x_2-\theta{t}}g(t,x,u(t,x))\,dxdt\\
			&-\int\limits_{x_1}^{x_1+\theta\bar{t}}\Bigg\{\left[u_1\left(\frac{y-x_1}{\theta},y\right)+u_2\left(\frac{y-x_1}{\theta},y\right)\right]\\
			&\quad\quad-\frac{1}{\theta}\left[f_1\left(u\left(\frac{y-x_1}{\theta},y\right)\right)+f_2\left(u\left(\frac{y-x_1}{\theta},y\right)\right)\right]\Bigg\}\,dy\\
			&-\int\limits^{x_2}_{x_2-\theta\bar{t}}\Bigg\{\left[u_1\left(\frac{x_2-y}{\theta},y\right)+u_2\left(\frac{x_2-y}{\theta},y\right)\right]\\
			&\quad\quad-\frac{1}{\theta}\left[f_1\left(u\left(\frac{x_2-y}{\theta},y\right)\right)+f_2\left(u\left(\frac{x_2-y}{\theta},y\right)\right)\right]\Bigg\}\,dy.
		\end{align*} 
		We can check that $\abs{f_1}+\abs{f_2}\leq C(\abs{u_1}+\abs{u_2})$. Then we have
		\begin{align*}
			&\left[u_1\left(\frac{y-x_1}{\theta},y\right)+u_2\left(\frac{y-x_1}{\theta},y\right)\right]-\frac{1}{\theta}\left[f_1\left(u\left(\frac{y-x_1}{\theta},y\right)\right)+f_2\left(u\left(\frac{y-x_1}{\theta},y\right)\right)\right]\\
			&\geq \left[\abs{u_1\left(\frac{y-x_1}{\theta},y\right)}+\abs{u_2\left(\frac{y-x_1}{\theta},y\right)}\right]\left(1-\frac{C}{\theta}\right)-2\mathcal{K}\eta,
		\end{align*}
		and similarly,
		\begin{align*}
			&\left[u_1\left(\frac{x_2-y}{\theta},y\right)+u_2\left(\frac{x_2-y}{\theta},y\right)\right]-\frac{1}{\theta}\left[f_1\left(u\left(\frac{x_2-y}{\theta},y\right)\right)+f_2\left(u\left(\frac{x_2-y}{\theta},y\right)\right)\right]\\
			&\geq \left[\abs{u_1\left(\frac{x_2-y}{\theta},y\right)}+\abs{u_2\left(\frac{x_2-y}{\theta},y\right)}\right]\left(1-\frac{C}{\theta}\right)-2\mathcal{K}\eta.
		\end{align*}
		Set $\theta=C$. Then it follows
		\begin{align*}
			&\int\limits_{I}\left[\abs{u_1(\bar{t},x)}+\abs{u_2(\bar{t},x)}-2\mathcal{K}\eta\right]\,dx\\
			&\leq \int\limits_{I_0}[u_1(0,x)+u_2(0,x)]\,dx+\int\limits_{0}^{\bar{t}}\int\limits_{x_1+\theta{t}}^{x_2-\theta{t}}g(t,x,u(t,x))\,dxdt+4\mathcal{K}C\bar{t}\eta\\
			&\leq \int\limits_{I_0}[u_1(0,x)+u_2(0,x)]\,dx+\abs{I}\norm{\omega_2}_{L^1(0,\bar{t})}+4\mathcal{K}C\bar{t}\eta,\\
			&\leq \int\limits_{I_0}[u_1(0,x)+u_2(0,x)]\,dx+l\eta+4\mathcal{K}C\bar{t}\eta.
		\end{align*}
		Hence,
		\begin{equation}
			\int\limits_{I}\left[\abs{u_1(\bar{t},x)}+\abs{u_2(\bar{t},x)}\right]\,dx\leq C^\p\eta(l+C^{\p\p}\bar{t}).
		\end{equation}
	\end{proof}
	
	Before we consider the next case, we would like to recall the following lemma from \cite{Bi-Col-Mon}.
	\begin{lemma}\label{lemma-tech-3}
		Let $u$ be a weak entropy solution constructed as in Proposition \ref{prop-construction-1} and $\{y_m(t)\}_{m\in\N}$ be the countable family of shock curves. Consider $\mathcal{G}(T,X):=\{\varphi\in W^{1,\f}[0,T]:\varphi(T)=X\}$ and $J:=\bigcup\limits_{m}\mbox{graph}(y_m)$. Then $\mathcal{F}:=\{\varphi\in \mathcal{G}:\,\mathcal{H}^1(\mbox{graph}(\varphi)\cap J)=0\}$ is dense in $\mathcal{G}(T,X)$ w.r.t the norm of $W^{1,\f}$, that is, $\norm{\varphi}_{W^{1,\f}}:=\norm{\varphi}_{L^\f}+\norm{\varphi^\p}_{L^\f}$.
	\end{lemma}
	Now we show an analogous result of Proposition \ref{proposition-int-1} when $f$ satisfies $\frac{\pa^2 f_1}{\pa u_2^2}(0)= 0$ and $\frac{\pa^2 f_2}{\pa u_1^2}(0)= 0$. More precisely, we show the following result.
	\begin{proposition}\label{proposition-int-2}
		Let $v$ be the solution constructed in section \ref{sec:construction} corresponding to an initial data verifying \eqref{smallness-condition-1} and \eqref{assumption-v-1}. If $\frac{\pa^2 f_1}{\pa u_2^2}(0)= 0$ and $\frac{\pa^2 f_2}{\pa u_1^2}(0)= 0$, then for any interval $I$ of length $l$ and $\bar{t}\geq0$, 
		\begin{equation}\label{est-int-2}
			\abs{\int\limits_{I}v_{i}(\bar{t},x)\,dx}\leq C^\p\eta(l+C^{\p\p}\bar{t})+C\norm{v(\bar{t})}^3_{L^\f}\bar{t}.
		\end{equation}
	\end{proposition}
	\begin{proof}
		Let $I=[x_l,x_r]$. Then let us consider two curves $x_l(t)$ and $x_r(t)$ such that $x_l(\bar{t})=x_l, \,x_r(\bar{t})=x_r$. Then using divergence theorem we obtain
		\begin{align}
			\int\limits_{I}u_i(\bar{t},x)\,dx&=\int\limits_{x_l(0)}^{x_r(0)}u_i(0,x)\,dx+\int\limits_{0}^{\bar{t}}[f_{i}(u(t,x_l(t)))-\dot{x}_l(t)u_i(t,x_l(t))]dt\\
			&\quad-\int\limits_{0}^{\bar{t}}[f_{i}(u(t,x_r(t)))-\dot{x}_r(t)u_i(t,x_r(t))]dt+\int\limits_{0}^{\bar{t}}\int\limits_{x_l(t)}^{x_r}g(x,u(t,x))\,dxdt
		\end{align}
		for $i=1,2$. Let $x_l(t)$ and $x_r(t)$ solve the following equations respectively
		\begin{equation}
			\dot{x}_l(t)=\frac{f_i(u(t,x_l(t)))}{u_i(t,x_l(t))}\mbox{ and }\dot{x}_r(t)=\frac{f_i(u(t,x_r(t)))}{u_i(t,x_r(t))}.
		\end{equation}
		By Lemma \ref{lemma-tech-3} there exists two Lipschitz curves $\hat{x}_l(t)$ and $\hat{x}_r(t)$ such that $\hat{x}_l(\bar{t})=x_l$, $\hat{x}_r(\bar{t})=x_r$ and
		\begin{equation}
			\norm{\dot{x}_l-\dot{\hat{x}}_l}_{L^\f}\leq \norm{u}^2_{L^\f},\quad\norm{\dot{x}_r-\dot{\hat{x}}_r}_{L^\f}\leq \norm{u}^2_{L^\f}.
		\end{equation}
		Since $\abs{x_r-x_l(t)}\leq l+C_6\bar{t}$ for all $t\in[0,\bar{t}]$, we have
		\begin{equation*}
			\int\limits_{0}^{\bar{t}}\int\limits_{x_l(t)}^{x_r}g(x,u(t,x))\,dxdt\leq (l+C_6\bar{t})\norm{\omega_2}_{L^1(0,\bar{t})}.
		\end{equation*} Hence the estimate \eqref{est-int-2} follows.
	\end{proof}
	\begin{proposition}\label{proposition-int-3}
		Let $v$ be the solution constructed in section \ref{sec:construction} corresponding to an initial data verifying \eqref{smallness-condition-1} and \eqref{assumption-v-1}. If $\frac{\pa^2 f_1}{\pa u_2^2}(0)\neq 0$ and $\frac{\pa^2 f_2}{\pa u_1^2}(0)= 0$ (or $\frac{\pa^2 f_1}{\pa u_2^2}(0)= 0$ and $\frac{\pa^2 f_2}{\pa u_1^2}(0)\neq 0$), then for any interval $I$ of length $l$ and $\bar{t}\geq0$, 
		\begin{equation}\label{est-int-3}
			\abs{\int\limits_{I}v_{i}(\bar{t},x)\,dx}\leq C^\p\eta(l+C^{\p\p}\bar{t})+C\norm{v(\bar{t})}^3_{L^\f}\bar{t}.
		\end{equation}
	\end{proposition}
	
	\begin{proof}
		We consider the situation when $\frac{\pa^2 f_1}{\pa u_2^2}(0)\neq 0$ and $\frac{\pa^2 f_2}{\pa u_1^2}(0)= 0$. By Lemma \ref{lemma-int-1} we get
		\begin{equation*}
			\abs{u_1}\leq u_1+2\mathcal{K}\eta.
		\end{equation*}
		As in Proposition \ref{proposition-int-1}, it follows that
		\begin{equation*}
			\abs{\int\limits_{I}u_1(t,x)\,dx}\leq C^\p\eta(l+C^{\p\p}\bar{t}).
		\end{equation*}
		By a similar argument as in Proposition \ref{proposition-int-2} we can obtain
		\begin{equation}
			\int\limits_{I}\abs{u_2(t,x)}\,dx\leq C^\p\eta(l+C^{\p\p}\bar{t})+C\norm{u(\bar{t})}_{L^\f}^3.
		\end{equation}
	\end{proof}

We note that for $2\times 2$ conservation laws, the case $\frac{\pa^2 f_1}{\pa u_2^2}(0)\neq 0$ and $\frac{\pa^2 f_2}{\pa u_1^2}(0)\neq 0$ is covered in the work by Glimm and Lax \cite{Glimm-Lax}. The other two cases are previously done in \cite{Bi-Col-Mon}. In Proposition \ref{proposition-int-1}, \ref{proposition-int-2} and \ref{proposition-int-3}, we extend the results of \cite{Bi-Col-Mon,Glimm-Lax} in the $2\times 2$ system  balance laws.
	
	\begin{proposition}\label{proposition-L-infty}
		There exists $K>0$ such that for any initial data $\bar{v}$ verifying \eqref{smallness-condition-1} and for all $m\in\N$ the solution to \eqref{eqn-main} satisfies 
		\begin{equation}
			\norm{v(t_m)}_{L^\f}\leq K\sqrt{\eta},
		\end{equation}
		where $t_m$ is as in \eqref{defn-tm}.
	\end{proposition}
	
	\begin{proof}
		We prove by induction. For $m=0$, it holds for any $K>\sqrt{\eta}$. Now suppose that the result holds for $k\leq m-1$.
		
		Note that the lower basis of $\bigtriangleup_m$ is covered by the upper basis of $4(m-1)$ trapezoids and $\mathcal{T}$ be the union of them. We divide $\mathcal{T}_{m-1}$ by horizontal segments $b_{m-1}^0,\cdots,b_{m-1}^N$ into $N$ sub trapezoids, let us denote them by $\mathcal{T}_{m-1}^1,\cdots,\mathcal{T}_{m-1}^N$. Note that $\mathcal{T}^j_{m-1}$ has height $h_N=2^{m-2}L/(N\hat{\la})$, upper basis $b_{m-1}^{j}$ and lower basis $b_{m-1}^{j-1}$ for $j=1,\cdots, N$. Note that
		\begin{align*}
			&\sum\limits_{k=1}^{N}(\hat{Q}(v(t_{m-1}+(k-1)h_N)|_{b_{m-1}^{k-1}})-\hat{Q}(v(t_{m-1}+kh_N)|_{b_{m-1}^{k}}))\\
			&= \hat{Q}(v(t_{m-1})|_{b_{m-1}^{0}})-Q(v(t_{m-1}+(k-1)h_N)|_{b_{m-1}^{N}}). 
		\end{align*} 
		Since $(\hat{Q}(v(t_{m-1}+(k-1)h_N)|_{b_{m-1}^{k-1}})-\hat{Q}(v(t_{m-1}+kh_N)|_{b_{m-1}^{k}}))\geq0$ for all $k=1,\cdots,N$, there exists $n\in\{1,\cdots,N\}$ such that
		\begin{align*}
			&(\hat{Q}(v(t_{m-1}+(n-1)h_N)|_{b_{m-1}^{n-1}})-\hat{Q}(v(t_{m-1}+nh_N)|_{b_{m-1}^{n}}))\\
			&\leq \frac{1}{N}\left[\hat{Q}(v(t_{m-1})|_{b_{m-1}^{0}})-\hat{Q}(v(t_{m-1}+(k-1)h_N)|_{b_{m-1}^{N}})\right].
		\end{align*}
		Hence, we get
		\begin{align*}
			&(\hat{Q}(v(t_{m-1}+(n-1)h_N)|_{b_{m-1}^{n-1}})-\hat{Q}(v(t_{m-1}+nh_N)|_{b_{m-1}^{n}}))\\
			&\leq \frac{1}{N}\hat{Q}(v(t_{m-1})|_{b_{m-1}^{0}})\\
			&\leq \frac{1}{N}(TV(v(t_{m-1})))^2\\
			&\leq \frac{1}{N}\left(\frac{20\hat{\la}}{c_0}\right)^2.
		\end{align*}
		Fix two points $(t,x)$ and $(t,y)$ on the segment $b_{m-1}^{n}$ (say $=[\underline{b},\overline{b}]$) with $x<y$. By using Theorem \ref{thm-construction-1}, we have
		\begin{equation*}
			v_i(t,y)\leq v_i(t,x)+\frac{N}{L}\frac{y-x}{2^{m-2}}\frac{\hat{\la}}{c_0}+\frac{\mathcal{M}}{N}\left(\frac{20\hat{\la}}{c_0}\right)^2\left(\norm{v(t_{m-1})}_{L^\f}+\norm{\omega_2}_{L^1(0,T^*)}\right).
		\end{equation*}
		Integrating w.r.t $y$ we get
		\begin{equation*}
			\frac{1}{l}\int\limits_{x}^{x+l}v_i(t,y)\,dy\leq v_i(t,x)+\frac{N}{L}\frac{l}{2^{m-1}}\frac{\hat{\la}}{c_0}+\frac{\mathcal{M}}{N}\left(\frac{20\hat{\la}}{c_0}\right)^2\left(\norm{v(t_{m-1})}_{L^\f}+\norm{\omega_2}_{L^1(0,T^*)}\right).
		\end{equation*}
		Similarly, integrating over $(y-l,y)$ we obtain
		\begin{equation*}
			v_i(t,y)\leq 	\frac{1}{l}\int\limits_{y-l}^{y}v_i(t,x)\,dx+\frac{N}{L}\frac{l}{2^{m-1}}\frac{\hat{\la}}{c_0}+\frac{\mathcal{M}}{N}\left(\frac{20\hat{\la}}{c_0}\right)^2\left(\norm{v(t_{m-1})}_{L^\f}+\norm{\omega_2}_{L^1(0,T^*)}\right).
		\end{equation*}
		Hence, we have for $x\in [\underline{b},\overline{b}-l]$ with $l<\frac{\overline{b}-\underline{b}}{2}$,
		\begin{equation*}
			\abs{v_i(t,x)}\leq \frac{1}{l}\abs{\int\limits_{y-l}^{y}v_i(t,z)\,dz}+\frac{N}{L}\frac{l}{2^{m-1}}\frac{\hat{\la}}{c_0}+\frac{\mathcal{M}}{N}\left(\frac{20\hat{\la}}{c_0}\right)^2\left(\norm{v(t_{m-1})}_{L^\f}+\norm{\omega_2}_{L^1(0,T^*)}\right).
		\end{equation*}
	for any $y\in [x,\overline{b}]$. From Proposition \ref{proposition-int-1}, \ref{proposition-int-2} and \ref{proposition-int-3} we obtain
		\begin{align*}
			\abs{v_i(t,x)}\leq &C^\p\left(l+C^{\p\p}\frac{t}{l}\right)+C\norm{v(t)}_{L^\f}^3\frac{t}{l}+\frac{N}{L}\frac{l}{2^{m-1}}\frac{\hat{\la}}{c_0}\\
			&+\frac{\mathcal{M}}{N}\left(\frac{20\hat{\la}}{c_0}\right)^2(\norm{v(t_{m-1})}_{L^\f}+\norm{\omega}_{L^1}).
		\end{align*}
		Set $l/t=\sqrt{\eta+\norm{v(t)}_{L^\f}}$. By induction hypothesis $\norm{v(t_{m-1})}_{L^\f}\leq K\sqrt{\eta}$. By Proposition \ref{prop-construction-1} we get $\norm{v(t)}_{L^\f}\leq MK\sqrt{\eta}$. Subsequently,
		\begin{align*}
			\norm{v(t)}_{L^\f}&\leq C(\eta+\sqrt{\eta})+C\norm{v(t)}_{L^\f}^{\frac{3}{2}}+\frac{N}{c_0}\frac{\sqrt{\eta+\norm{v(t)}_{L^\f}^3}}{2^{m-1}}\frac{\hat{\la}t}{L}\\
			&+\frac{\mathcal{M}}{N}\left(\frac{20\hat{\la}}{c_0}\right)^2(MK\sqrt{\eta}+\eta)\\
			&\leq C_*N\sqrt{\eta}+\frac{C^*}{N}\sqrt{\eta}.
		\end{align*}
		Setting $N=4C_*M$ and $K=4C_*MN$, we have
		\begin{equation}
			\norm{v(t)}_{L^\f}\leq \frac{K}{4M}\sqrt{\eta}+\frac{1}{4M}\sqrt{\eta}\leq \frac{K}{2M}\sqrt{\eta}.
		\end{equation}
		By Proposition \ref{prop-construction-1} we obtain $\norm{v(t_m)}_{L^\f}\leq M\norm{v(t)}_{L^\f}\leq \frac{K}{2}\sqrt{\eta}$. This completes the proof of Proposition \ref{proposition-L-infty}.
	\end{proof}

\section{Structure of solution}\label{sec:structure}

\begin{theorem}
	Let $u$ be a weak solution to \eqref{eqn-main}--\eqref{eqn-data} as a limit of front tracking approximation. Then there exist a countable family of Lipschitz continuous curves $\Gamma=\{(t,y_m(t));t\in(t^-_m,t^+_m)\}$ and a countable set $\Theta\subset[0,\f)\times\R$ such that the following holds. For each $m\geq1$ and each $t\in(t_m^-,t_m^+)$ with $(t,y_m(t))\notin \Theta$, there exists derivative $\dot{y}_m(t)$ and the left and right limits exist,
	\begin{equation}\label{solution-trace}
		u^-:=\lim\limits_{\tiny \begin{array}{rc}
			 &(s,z)\rr (t,y_{m}(t)),\\
			 &y<y_m(t)
			\end{array}}u(s,z)\mbox{ and }\,\,
		u^+:=\lim\limits_{\tiny \begin{array}{rc}
				&(s,z)\rr (t,y_{m}(t)),\\
				&y>y_m(t)
		\end{array}}u(s,z).
	\end{equation} 
   These limits satisfy the Rankine-Hugoniot equations and the lax conditions 
   \begin{equation}\label{solution-R-H}
   	(u^+-u^-)\dot{y}_m=(f(u^+)-f(u^-))\mbox{ and }\la_i(u^+)\leq \dot{y}_m\leq \la_i(u^-)\mbox{ for some }i\in\{1,2\}.
   \end{equation} 
Moreover, $u$ is continuous outside the set $\Theta\cup\Gamma$.
\end{theorem}
\begin{proof}
	The proof is divided into several steps. 
	\noindent
	\begin{enumerate}
		\item[Step-1.] Fix a sequence $\e_\nu\rr0$. For each $\nu$ let $u_\nu $ be the approximate solution obtained from front tracking algorithm. Set $\mu_{\nu,0}^{i\pm}$ as measure of the positive and negative $i$-waves present in th piecewise constant initial data. As $\nu\rr\f$, we assume that
		\begin{equation*}
			\mu^{i\pm}_{\nu,0}\rightharpoonup \bar{\mu}^{i\pm}\mbox{ for }i=1,2.
		\end{equation*}  
	Define $\mu^{IC}_{\nu}$ as follows
	\begin{equation}
		\mu_\nu^{IC}(\{(t,x)\}):=\left\{\begin{array}{cl}\abs{\si\si^{\p}}+\left\{\begin{array}{cl}
			\abs{\si}+\abs{\si^\p}-\abs{\si+\si^\p}&\mbox{ if }i=i^\p,\\
			0&\mbox{ if }i\neq i^\p,
		\end{array}\right.&\mbox{ for }t\neq t_n,\\
	  \left\{\begin{array}{cl}
	  \tau_\nu\omega_{2,n}(\abs{\si}+\e_\nu\omega_{1,j})&\mbox{ if }x=x_j,\\
	  	\tau_\nu\omega_{2,n}\abs{\si}&\mbox{ if }x\neq x_j,
	  \end{array}\right.&\mbox{ for }t= t_n.
	\end{array}\right.
	\end{equation}
   We note that $\mu_\nu^{IC}([0,T^*]\times\R)\leq C_*$ for some $C_*>0$ which is independent of $\nu$. Hence, we may assume up to a subsequence that
   \begin{equation}
   	\mu^{IC}_\nu\rightharpoonup \bar{\mu}^{IC}\mbox{ as }\nu\rr\f.
   \end{equation}
   We consider two countable sets $\Theta_0,\Theta_1$ as follows
   \begin{equation}
   	\Theta_0:=\left\{(0,x);\bar{u}(x+)\neq \bar{u}(x-)\right\}\mbox{ and }\Theta_1:=\left\{(t,x);\bar{\mu}^{IC}(\{(t,x)\})>0\right\}.
   \end{equation} 
  
  \item[Step-2]We define the $i$-shock curves in the solutions $u$. Fix $\vartheta>0$. By a $\vartheta$-shock front of the $i$-th family in the approximate solution $u_\nu$, we mean a polygonal line in $t-x$ plane with nodes $(s_0,z_0),(s_1,z_1),\dots,(s_N,z_N)$ satisfying 
  \begin{enumerate}
  	\item The points $(s_k,z_k)$ are interaction points with $0\leq s_0<s_1<\cdots<s_N$.
  	\item For each $k$ the line joining $(s_{k-1},z_{k-1})$ and ($s_k,z_k)$ is an $i$-shock with strength $\abs{\si_k}\geq\vartheta/2$. Moreover, $\abs{\si_{k_*}}\geq\vartheta$ for some $k_*\in\{1,\cdots,N\}$.
  	\item For $k<N$, if two incoming $i$-shocks both of strengths $\geq \vartheta/2$ interact at the node $(s_k,z_k)$, the shock coming from $(s_{k-1},z_{k-1})$ has the largest speed.  
  \end{enumerate} 
   We call a $\vartheta$-shock as maximal $\vartheta$-shock if it is maximal w.r.t. set theoretical inclusion. We note that the number of maximal $\vartheta$-shock starting from $s_0=0$ is bounded by $2\vartheta^{-1}V(\bar{u}_\nu)$. Observe that a shock curve can go through four types of interactions.
   \begin{enumerate}
   	\item The $i$-shock curve interacts with $j$-shock for $j\neq i$.
   	\item The $i$-shock curve interacts with $i$-shock.
   	\item The node $(s_k,z_k)$ occurs at $t=t_n$ for some $n$ but $z_k\neq x_j$ for all $j\in\mathbb{Z}$. 
   	\item The node $(s_k,z_k)$ occurs at $t=t_n$ for some $n$ and $z_k= x_j$ for some $j\in\mathbb{Z}$. 
   \end{enumerate}
 Suppose the $\vartheta$-shock curve starting from $t=t_0>0$ achieves the strength $\geq \vartheta$ first time at $t=t_{k_*}$. Let $\mathcal{I}_{a},\mathcal{I}_b,\mathcal{I}_c,\mathcal{I}_d$ be disjoint subsets of $\{0,1,\cdots,k_*\}$ containing indices such that the corresponding node of the type $a,\, b,\, c,\, d$ respectively. We obtain
 \begin{align}
 	\frac{\vartheta}{2}\leq & C_3\sum\limits_{k\in \mathcal{I}_a}\left[\abs{\si^\p_{i,k}}+\abs{\si_{i,k}\si_{i,k}^\p}(\abs{\si_{i,k}}+\abs{\si^\p_{i,k}})\right]\\
 	&+C_3\sum\limits_{k\in \mathcal{I}_b}\abs{\si_{i,k}\si_{j,k}}(\abs{\si_{i,k}}+\abs{\si_{j,k}})\\
 	&+C_3\sum\limits_{k\in \mathcal{I}_c}\tau_\nu\omega_{2,n_k}\abs{\si_{i,k}}+C_3\sum\limits_{k\in \mathcal{I}_d}\tau_\nu\omega_{2,n_k}\e_\nu\omega_{1,j_k}.
 \end{align}
Hence, the functional $\hat{Q}$ decays in this process at least by an amount $\geq c_*K\vartheta^3$. Therefore, we have
\begin{equation}
	N_\nu:=\#\left\{\mbox{maximal $\vartheta$-shock fronts in $u_\nu$}\right\}\leq (c_*K\vartheta^3)^{-1}\hat{Q}(T^*)+2\vartheta^{-1}V(\bar{u}_\nu).
\end{equation}
Let $\{y_{m,\nu}\}_{m=1}^{N_\nu}$ be the set of $\vartheta$-maximal shock curves for $u_\nu$ such that $y_{m,\nu}:[t^-_{m,\nu},t_{m,\nu}^+]\rr\R$ for some $t_{m,\nu}^-<t_{m,\nu}^+$. Now, up to subsequence, we assume that $N_\nu=\bar{N}$ for all $\nu$. Then, up to subsequence we can assume that
\begin{equation}
	y_{m,\nu}(\cdot)\rr y_m(\cdot)\mbox{ and }t_{m,\nu}^\pm\rr t^\pm_m\mbox{ for each }m\in\{1,\cdots,\bar{N}\}
\end{equation}
for some Lipschitz continuous curve $y_{m}$. Finally by setting $\vartheta=1,1/2,1/3,\cdots$ we obtain a countable family of Lipschitz curves $\{y_m\}=:\Gamma$ such that $y_m:(t_m^-,t_m^+)$ for some $t_m^-<t_m^+$. 

\item[Step-3.] Let $\Theta_2$ be the set of intersection points of $y_m$ and $y_n$ for all $y_m,y_n\in\Gamma$. We consider
\begin{equation}
	\Theta:=\Theta_0\cup\Theta_1\cup\Theta_2.
\end{equation}

\item[Step-4.] Let $P=(\tau,\xi)=(\tau,y_m(\tau))$ be a point on $y_m$ such that $P\notin \Theta$. Since $u(\tau,\cdot)$ is a BV function, there exists two states $u^\pm$ such that
\begin{equation}
	u^-=\lim\limits_{x\rr y_m(\tau)-}u(\tau,x)\mbox{ and }u^+=\lim\limits_{x\rr y_m(\tau)+}u(\tau,x).
\end{equation}
We are going to show that \eqref{solution-trace} holds true for $u^\pm$. By the construction there exists a sequence $y_{m,\nu}\rr y_m$. We assume that for each $\nu$, $y_{m,\nu}$ is an $\vartheta$-shock curve corresponding to $i$-th family. We claim that
\begin{align}
	\lim\limits_{r\rr0}\limsup\limits_{\nu\rr\f}\left(\sup\limits_{\tiny\begin{array}{rc}
			&x<y_{m}(t)\\
			&(t,x)\in B(P,r)
		\end{array}}\abs{u_\nu(t,x)-u^-}\right)&=0,\label{limit-u-}\\
		\lim\limits_{r\rr0}\limsup\limits_{\nu\rr\f}\left(\sup\limits_{\tiny\begin{array}{rc}
			&x>y_{m}(t)\\
			&(t,x)\in B(P,r)
	\end{array}}\abs{u_\nu(t,x)-u^+}\right)&=0.\label{limit-u+}
\end{align}
We argue by contradiction. Suppose \eqref{limit-u-} does not hold. Then there exists $\epsilon_0>0$ a sequence of points $Q_\nu$ such that 
\begin{equation}
	Q_\nu\rr P, \,\, \abs{u(Q_\nu)-u^-}\geq \epsilon_0\mbox{ for all }\nu\geq1.
\end{equation}
Furthermore, we can find another sequence of points $\{P_\nu\}$ such that the segment $P_\nu Q_\nu$ is space-like and 
\begin{equation}
	P_\nu\rr P\mbox{ and }u_\nu(P_\nu)\rr u^-\mbox{ as }\nu\rr\f.
\end{equation}
To fix the ideas we consider $i=1$. 
\begin{enumerate}

	\item Suppose in the solution $u_\nu$, the segment $P_\nu Q_\nu$ is crossed by an amount $\geq \epsilon_1$ of $2$-waves in $u_\nu$ for some $\epsilon_1>0$ independent of $\nu$. 
	
	Let $\Gamma_\nu$ be the region bounded by the 1-shock $y_{m,\nu}$, the segment $P_\nu Q_\nu$, the minimal forward 2-characteristic from $P_\nu$ (intersecting $y_m$ at $A_\nu$), the maximal forward 2-characteristic from  $Q_\nu$ (intersecting $y_m$ at $B_\nu$).  Note that no 2-wave can leave $\Gamma_\nu$ from the sides $P_\nu A_\nu$ and $Q_\nu B_\nu$. Therefore, 2-waves either get canceled by interactions or cross the 1-shock curve $y_m$. Hence, there exists $c>0$ such that either $\mu_\nu^{IC}(\Gamma_\nu)\geq c_0\epsilon_1^2$ or 2-waves of total amount $\geq c\epsilon_1$ crosses $y_m$. In the later case we obtain $\mu_\nu^{IC}(\bar{\Gamma}_\nu)\geq c \epsilon_1\vartheta$. This gives a contradiction.
	
		\item Suppose in the solution $u_\nu$, the segment $P_\nu Q_\nu$ is crossed by an amount $\geq \epsilon_1$ of $1$-waves in $u_\nu$ for some $\epsilon_1>0$ independent of $\nu$ and the total amount of $2$-waves crossing $P_\nu Q_\nu$ vanishes to zero as $\nu\rr\f$.
		
		In this case, we construct the region $\Gamma_\nu$ as follows. We extend the segment $P_\nu Q_\nu$ on right and let it intersects $y_m$ at $B_\nu$. Now, we would like to choose a point $C_\nu$ such that the segment $C_\nu B_\nu$ contains 1-wave of amount $\geq \vartheta/4$. We consider the minimal forward 1-characteristic $x(t)$ from $C_\nu$. Note that the slope $\dot{x}>\dot{y}_{m,\nu}+c^\p\vartheta$ for some $c^\p>0$ due genuinely non-linearity. In this case, $x(t)$ hits $y_{m,\nu}$ at some point $A_\nu$. Then we consider $\Gamma_\nu$ as the region enclosed by $C_\nu B_\nu$, $x(t)$ and $y_{m,\nu}(t)$. Furthermore, we have
		\begin{equation}
			\abs{P_\nu-A_{\nu}}\leq \frac{C_\nu-B_\nu}{c^\p\vartheta}\rr0\mbox{ as }\nu\rr\f.
		\end{equation}
		Since no 1-wave can exit $\Gamma_\nu$ we have $\mu_\nu^{IC}(\bar{\Gamma}_\nu)\geq \vartheta^2$. This implies $\bar{\mu}^{IC}(\{P\})>0$.

\end{enumerate}
Above cases cover all the situations. Hence, we establish the limit \eqref{limit-u-}. By a similar argument \eqref{limit-u-} follows. By a similar argument as in \cite{Bressan-book} we obtain \eqref{solution-R-H}.

\item[Step-5.] Consider a point $P=(\tau,\xi)\notin \Theta$ and $P\notin y_{m}(t)$ for any $m$. To show that $u$ is continuous at $P$ we argue by contradiction. Suppose $u$ is not continuous at $P$, Then there exists $\epsilon_0>0$ and $P_\nu=(s_\nu^P,z_\nu^P), Q_\nu=(s_\nu^Q,z_\nu^Q)$ such that $Q_\nu P_\nu$ is space-like and $P_\nu,Q_\nu\rr P$ with 
\begin{equation}
	\abs{u_\nu(Q_\nu)-u(P)}\geq \e_0 \mbox{ for all }\nu\geq1\mbox{ and }u_\nu(P_\nu)\rr u(P)\mbox{ as }\nu\rr\f.
\end{equation}
Now we have the following two possibilities.
\begin{enumerate}
	\item Suppose there exists $\epsilon_1>0$ such that amount of 1-wave crossing $P_\nu Q_\nu$ is $\geq \epsilon_1$ and amount of 2-wave crossing $P_\nu Q_\nu$ is $\geq \e_1$. Set $s_\nu^*=\min\{s_\nu^P,s_\nu^Q\}$ and consider the box $\Gamma_\nu$ defined as follows
	\begin{equation}
		\Gamma_\nu=\left\{(t,x);\abs{t-s_\nu^*}\leq \rho_\nu,x\in [\min\{z_\nu^P,z_\nu^Q\},\max\{z_\nu^P,z_\nu^Q\}]\right\},
	\end{equation}
with $\rho_\nu=\abs{P_\nu-Q_\nu}+2\frac{\abs{P_\nu-Q_\nu}}{c_0}$ where $c_0$ is as in \eqref{def:strict-hyperbolicity}. Note that $\mu_\nu^{IC}(\Gamma_\nu)\geq \epsilon_1^2$, Since $\rho_\nu\rr0$ as $\nu\rr\f$, we obtain $\bar{\mu}^{IC}(\{P\})>0$. This gives a contradiction.
    
    \item Suppose there exists $\epsilon_1>0$ such that amount of 1-wave crossing $P_\nu Q_\nu$ is $\geq \epsilon_1$ and the total amount of 2-wave crossing $P_\nu Q_\nu$ is $\geq \e_1$ converges to zero as $\nu\rr\f$.
    
    Note that for large enough $\nu$ we assume that $\abs{v_2(u_\nu(P_\nu))-v_2(u_\nu(Q_\nu))}\leq \epsilon_1^4$ and $\abs{v_1(u_\nu(P_\nu))-v_1(u_\nu(Q_\nu))}\geq \epsilon_1$. For sufficiently small $\epsilon_1$ we can conclude that $\abs{\la_1(u_\nu(P_\nu))-\la_1(u_\nu(P_\nu))}\geq \epsilon_2$ for some $\epsilon_2>0$ independent of $\nu$. Now, set $s_\nu^*=\min\{s_\nu^P,s_\nu^Q\}$ and consider the box $\Gamma_\nu$ defined as follows
    \begin{equation}
    	\Gamma_\nu=\left\{(t,x);\abs{t-s_\nu^*}\leq \rho_\nu,x\in [\min\{z_\nu^P,z_\nu^Q\},\max\{z_\nu^P,z_\nu^Q\}]\right\},
    \end{equation}
    with $\rho_\nu=\abs{P_\nu-Q_\nu}+2\frac{\abs{P_\nu-Q_\nu}}{\epsilon_2}$. Furthermore observe that the maximum strength of 1-waves crossing $P_\nu Q_\nu$ converges to zero as $\nu\rr\f$. Indeed if there exists a 1-shock of strength $\geq\epsilon_3>0$ (independent of $\nu$) crossing $P_\nu Q_\nu$, then we know that the shock wave should converge to some $y_{m_0}(t)$. This gives a contradiction as $P\notin y_m$ for any $m$ and $P_\nu,Q_\nu\rr P$ as $\nu\rr\f$, it implies that for large enough $\nu$, $P_\nu Q_\nu$ can not be close to $y_{m_0}(t)$. Therefore, all the 1-waves either interact in future time inside $\Gamma_\nu$ or they come from interactions inside $\Gamma_\nu$ before crossing $P_\nu Q_\nu$. 
\end{enumerate}
	\end{enumerate}
\end{proof}

%
%
%
%

\end{document}